\documentclass[final,12pt]{colt2026} 

\title[Frequentist Regret Analysis of GP-TS via Fractional Posteriors]{Frequentist Regret Analysis of Gaussian Process Thompson Sampling via Fractional Posteriors}
\usepackage{times}
\usepackage{algorithm,algpseudocode,algorithm2e}

\makeatletter
\newcounter{eqcontrib}
\newcommand{\equalcontrib}{%
  \ifnum\value{eqcontrib}=0
    \thanks{These authors contributed equally.}\setcounter{eqcontrib}{\value{mpfootnote}}%
  \else
    \footnotemark[\value{eqcontrib}]%
  \fi
}
\makeatother

\coltauthor{%
 \Name{Somjit Roy}\equalcontrib
 \Email{sroy\_123@tamu.edu}\\
 \addr Department of Statistics\\
 Texas A\&M University\\
 College Station, TX 77843, USA
 \AND
 \Name{Prateek Jaiswal}\equalcontrib
 \Email{jaiswalp@purdue.edu}\\
 \addr Daniels School of Business\\
 Purdue University\\
 West Lafayette, IN 47907, USA
 \AND
 \Name{Anirban Bhattacharya} \Email{anirbanb@stat.tamu.edu}\\
 \addr Department of Statistics\\
 Texas A\&M University\\
 College Station, TX 77843, USA
 \AND
 \Name{Debdeep Pati}
 \Email{dpati2@wisc.edu}\\
 \addr Department of Statistics\\
 University of Wisconsin-Madison\\
 Madison, WI 53706, USA
 \AND
 \Name{Bani K. Mallick}
 \Email{bmallick@stat.tamu.edu}\\
 \addr Department of Statistics\\
 Texas A\&M University\\
 College Station, TX 77843, USA
}

\ifjmlrcleveref
  \newaliascnt{assumption}{theorem}
  \newtheorem{assumption}[assumption]{Assumption}
  \aliascntresetthe{assumption}
  \crefname{assumption}{assumption}{assumptions}
\else
  \newtheorem{assumption}[theorem]{Assumption}
\fi

\makeatletter

\renewcommand*{\@titlefoot}{} 

\def\ps@jmlrtps{%
  \let\@mkboth\@gobbletwo
  \def\@oddhead{\scriptsize
    \@jmlrproceedings
    \ifx\@jmlrworkshop\@empty\else\space\@jmlrworkshop\fi
    \hfill
  }%
  \let\@evenhead\@oddhead
  \def\@oddfoot{}%
  \let\@evenfoot\@oddfoot
}

\makeatother

\begin{document}

\maketitle

\begin{abstract}
We study Gaussian Process Thompson Sampling (GP-TS) for sequential decision-making over compact, continuous action spaces and provide a frequentist regret analysis based on fractional Gaussian process posteriors, without relying on domain discretization as in prior work. We show that the variance inflation commonly assumed in existing analyses of GP-TS can be interpreted as Thompson Sampling with respect to a fractional posterior with tempering parameter $\alpha \in (0,1)$. We derive a kernel-agnostic regret bound expressed in terms of the information gain parameter $\gamma_t$ and the posterior contraction rate $\epsilon_t$, and identify conditions on the Gaussian process prior under which $\epsilon_t$ can be controlled. As special cases of our general bound, we recover regret of order $\mathcal{\tilde O}(T^{\frac{1}{2}})$ for the squared exponential kernel, $\mathcal{\tilde O}(T^{\frac{2\nu+3d}{2(2\nu+d)}} )$ for the Mat\'ern-$\nu$ kernel, and a bound of order $\tilde{\mathcal O}(T^{\frac{2\nu+3d}{2(2\nu+d)}})$ for the rational quadratic kernel. Overall, our analysis provides a unified and discretization-free regret framework for GP-TS that applies broadly across kernel classes.
\end{abstract}


\section{Introduction}\label{sec:introduction}

Many sequential decision-making problems involve selecting actions from a potentially infinite decision space in order to optimize an unknown objective under a limited evaluation budget. We consider a setting in which the decision space is explored sequentially and performance is measured by how efficiently the algorithm identifies the optimal decision using a finite number of observations. This formulation captures a broad range of applications, including adaptive materials and experimental decision point, hyperparameter tuning and model selection in machine learning, simulation-based optimization, and control problems with expensive or noisy evaluations~\citep{frazier2018bayesian,Lei2021}.

To formalize such sequential decision-making problems over uncountable decision spaces, we assume the existence of an unknown objective function $\pmb\theta_0(\cdot)$ that maps each decision alternative $x \in \mathcal{X} \subset \mathbb{R}^d$ to a real-valued outcome. The goal of the decision-maker is to identify a point $x_0 \in \mathcal{X}$ that maximizes $\pmb\theta_0(\cdot)$. Under the assumption that $\pmb\theta_0(\cdot)$ belongs to a sufficiently regular function class, we study the decision point and analysis of sequential algorithms that use past observations at previously queried points to form a belief over the unknown function and leverage this belief to guide future decisions. Algorithmic performance is evaluated using the standard notion of cumulative pseudo-regret (or simply regret), which measures the total loss incurred from failing to select the globally optimal alternative at each iteration.

Specifically, we study Gaussian Process Thompson Sampling (GP-TS)~\citep{Ray&Gopalan2017,kandasamy18a}, a popular class of algorithms to solve such sequential decision-making problems. Existing analyses of GP-TS commonly introduce posterior variance inflation as a technical device to facilitate regret analysis~\citep{Ray&Gopalan2017}. However, this inflation is typically imposed in an ad hoc manner and lacks a formal Bayesian characterization. We show that variance-inflated GP-TS can be equivalently interpreted as Thompson Sampling with respect to a fractional Gaussian process (GP) posterior, obtained by tempering the likelihood with a factor $\alpha \in (0,1)$. This fractional posterior formulation, originating in the Bayesian statistics literature~\citep{bhattacharya2019bayesian}, enables variance inflation to be incorporated directly at the level of the posterior distribution. 

In this work, we provide a general framework for deriving frequentist regret bounds for GP-TS with variance inflation induced via fractional GP posteriors. Our analysis yields a kernel-agnostic regret bound that applies to GP-TS with any covariance kernel, provided two quantities are characterized: the information gain parameter $\gamma_t$ and the posterior contraction rate $\epsilon_t$ of the corresponding fractional GP posterior. While general techniques for bounding $\gamma_t$ based on kernel eigenvalue decay are available in prior work (most notably \cite{vakili21a}), we identify conditions on the Gaussian process prior distribution under which the contraction rate $\epsilon_t$ can be controlled, thereby completing a general recipe for computing regret bounds for GP-TS.

Our main regret bound is expressed explicitly in terms of $\gamma_t$ and $\epsilon_t$, and does not rely on discretization of the action space. As concrete instantiations of this general result, we specialize the bound to squared exponential and Mat\'ern-$\nu$ kernels under standard smoothness assumptions that the unknown objective function lies in the associated reproducing kernel Hilbert space (RKHS). For the squared exponential kernel, our general bound yields a cumulative regret of order $\mathcal{O}(T^{\frac{1}{2}}\log^{d+1} T)$. For the Mat\'ern-$\nu$ kernel, we obtain a regret bound of order $\mathcal{O}(T^{\frac{2\nu+3d}{2(2\nu+d)}} \log^{\frac{\nu}{2\nu+d} + \frac{q}{2+d}} T)$, where $q$ denotes a parameter governing the kernel scale. The resulting bound for the Mat\'{e}rn-$\nu$ kernel exhibits a gap of $\mathcal{O}(T^{\frac{d/2}{2\nu+d}})$ relative to the minimax rate. Moreover, our analysis also yields a regret bound for the rational quadratic kernel by combining the contraction analysis developed here with existing techniques to compute bounds on information gain from \cite{vakili21a}. These results are obtained by substituting kernel-specific expressions for $\gamma_T$ and $\epsilon_T$ into our general regret bound. More generally, the same technique can be applied to other kernels, highlighting the extensibility of our approach when combined with~\cite{vakili21a}'s results on information gain bounds. 

To derive our general regret bound, we need to assume that $\alpha t \epsilon_t^2 < 1$. In particular, this condition is needed to lower bound the probability of querying an unsaturated point (see Definition~\ref{def:US}) in the domain. This condition recovers the order of variance inflation commonly assumed in existing analyses of GP-TS, but arises here naturally within a kernel-agnostic analytical framework.

A key technical contribution of this work is a proof technique that eliminates the need for domain discretization, thereby overcoming a principal limitation of the analysis in~\cite{Ray&Gopalan2017}. Their approach derives high-probability regret bounds by constructing time-varying discretizations of the decision space. This is achieved by first establishing pointwise convergence results for GP posteriors~\citep{hoffman2013stochastic} and then lifting them to uniform guarantees via a union bound argument (see~\cite{Ray&Gopalan2017}[Lemma 5]), which necessarily restricts the analysis to a finite subset of the domain. In contrast, our analysis directly bounds the expected regret without constructing high-probability confidence sets or discretizing the domain. This is enabled by establishing posterior convergence guarantees for Gaussian processes in expectation, using analytical tools inspired by the Bayesian statistics literature on fractional posteriors~\citep{bhattacharya2019bayesian}. Moreover, our framework identifies a general technical condition on the Gaussian process prior that yields posterior concentration rates for the GP posterior in expectation. These concentration results form the foundation of our regret analysis and apply broadly across various kernel classes, thereby enabling a unified and discretization-free regret analysis of GP-TS.

The remainder of the article is organized as follows. Section \ref{sec:notation} introduces important notations used throughout, followed by a review of related literature in Section~\ref{sec:litrev}. Section~\ref{sec:Prob} introduces the problem setup, while Section~\ref{sec:reganal} develops the regret analysis and establishes frequentist regret bounds, with explicit instantiations for commonly used kernels provided in Section~\ref{subsec:bound-common-kernels}.
Section~\ref{subsec:first-order-concentration} characterizes the first-order concentration properties of the underlying $\alpha$-posterior distribution. Finally, we wrap up with a discussion of key findings and potential directions for future research in Section \ref{sec:conclusion}.

\subsection{Notations}~\label{sec:notation}
We write $\mathbb{R}^{d}$ for the $d$-dimensional real space, $\mathbb{Z}^{d}$ for the $d$-dimensional integer lattice, $\mathbb{R}^{+}$ for the set of positive reals, and $\mathbb{N}$ for the set of natural numbers. The notation $[T]:=\{1, 2, \ldots, T\}$ refers to the finite index set of the first $T$ rounds. For a finite set $A$, let $|A|$ denote its cardinality, and $\emptyset$ denotes the empty set. We use following non-asymptotic symbols: $f(T) = \mathcal{O}(g(T))$ indicates an upper bound up to a universal constant, $f(T) = \Omega(g(T))$ indicates a corresponding lower bound, and $f(T) = \tilde{\mathcal{O}}(g(T))$ denotes an upper bound up to polylogarithmic factors, i.e., $f(T) = \mathcal{O}(g(T)\mathrm{polylog}(T))$. The relation $a\asymp b$ means $a\lesssim b$ and $b\lesssim a$, where $\lesssim$ indicates inequality up to a universal constant. The space $L^{2}(\mathcal X)$ denotes the class of square-integrable functions on $\mathcal X$ viz., $L^{2}(\mathcal X) := \{f:\mathcal X\to \mathbb{R} \mid \int_{\mathcal X}|f(x)|^2dx < \infty\}$. The space of real-valued continuous functions on $[0, 1]^{d}$ is denoted by $\mathcal{C}[0, 1]^{d}$. We write $I_{d}$ for the identity matrix of order $d$ and $\mathcal{N}_{d}(\mu, \Sigma)$ for the $d$-variate Gaussian distribution with mean $\mu$ and covariance $\Sigma$. Inner products and norms are denoted as follows: $\langle \cdot, \cdot\rangle$ is the standard Euclidean inner product on $\mathbb{R}^{d}$, with $\lVert \cdot \rVert_{d}$ the corresponding Euclidean norm, and for a function $f$, the sup-norm is $\lVert f\rVert_{\infty} := \sup_{x\in \mathcal X}|f(x)|$. For a kernel $k(\cdot, \cdot)$, $\langle f, g\rangle_k$ is the inner product between functions $f$ and $g$ in the associated RKHS, and $\lVert f\rVert_{k}:= (\langle f, f\rangle_{k})^{\frac{1}{2}}$ is the RKHS norm. The Gamma function is denoted by $\Gamma(\cdot)$ and $K_{\nu}(\cdot)$ denotes the modified Bessel function of the second kind. We use $\Pi$ for a probability measure: $\Pi(\cdot)$ represents the prior distribution, $\Pi(\cdot\mid \mathcal H_t)$ the posterior after round $t$ given the observation history $\mathcal H_t$, and $\Pi_{ \alpha}(\cdot\mid \mathcal H_t)$ the $\alpha$-fractional posterior. For an event $A$, $\text{I}(A)$ denotes the indicator function taking value $1$ if $A$ occurs and $0$ otherwise. Finally, $\mathbb{P}_0^{(d)}$ denotes the $d$-fold product measure induced by the true data-generating distribution $p_0(y\mid x)$. 

\section{Related Work}~\label{sec:litrev}

Most theoretical work in Bayesian Optimization (BO) has focused on algorithms inspired by the principle of optimism in the face of uncertainty, most notably Upper Confidence Bound (UCB)–type methods. A canonical example is GP-UCB, introduced and analyzed in \citet{srinivas2010gaussian}, which establishes sublinear regret guarantees for GP optimization under suitable regularity conditions. Since then, a large body of work has developed UCB-style algorithms and refinements for GP-based optimization, and a comprehensive overview of these approaches can be found in \citet{Ray&Gopalan2017}. Beyond optimism-based methods, \citet{Ray&Gopalan2017} also introduced GP-TS, which assumes a time-varying Gaussian process prior over the unknown objective function and establishes a frequentist regret bound of order $\tilde{\mathcal O}(\gamma_T \sqrt{dT})$, where $\gamma_T$ denotes the maximum information gain from $T$ observations. However, their regret analysis relies on a time-dependent discretization of the action space $\mathcal X$ that becomes increasingly dense as $t \to \infty$. 

A complementary line of work studies regret bounds through the lens of kernel complexity and information-theoretic quantities. In particular, \citet{vakili21a} provide a general framework for bounding the information gain $\gamma_T$ based on the eigenvalue decay of the kernel integral operator. Using this approach, they show that for Mat\'ern-$\nu$ and squared exponential kernels, both GP-UCB and GP-TS achieve regret bounds of order
$
\mathcal O(T^{\frac{\nu+d}{2\nu+d}}\log^{\frac{4\nu+d}{4\nu+2d}}T)$ and $
\mathcal O(T^{\frac{1}{2}}\log^{\frac{1+d}{2}}T)
$
respectively. These rates are minimax-optimal up to logarithmic factors and match the lower bounds, $\Omega(T^{\frac{\nu+d}{2\nu+d}})$ and $\Omega(T^{\frac{1}{2}}\log^{\frac{d}{4}}T)$, derived in \citet{scarlett17a} for Mat\'ern-$\nu$ and squared exponential kernels.

More recent work has begun to examine the robustness of GP-based bandit algorithms to kernel misspecification and modeling errors. For instance, \citet{bogunovic2021misspecified} and \citet{kirschner2020distributionally} study GP bandits under various forms of model mismatch and distributional uncertainty, highlighting potential failure modes of overly confident posterior updates. Related efforts by \citet{calandriello2019gaussian} investigate adaptive strategies for kernel learning in GP bandits, aiming to mitigate misspecification by learning kernel parameters online. These works collectively suggest that algorithms achieving optimal regret under well-specified priors may still perform poorly when the assumed Gaussian process prior is even mildly incorrect. Finally, \citet{Kinjal2020} analyze Thompson Sampling for GP optimization from a Bayesian perspective, focusing on the convergence of the sequence of selected actions to the global optimum rather than cumulative regret minimization. While their results provide valuable insight into asymptotic behavior, they are not directly comparable to frequentist regret guarantees that dominate the BO literature.

Parallel to our work, \citet{li2026robustbayesianoptimizationtempered} study the use of fractional posteriors in a BO setting. While both works employ posterior tempering ideas, the problem setting and technical analyses differ, and the results are complementary.

\section{Problem Setup}~\label{sec:Prob}
We consider the problem of sequentially optimizing an unknown function over a compact decision space $\mathcal{X} = [0,1]^d$. An agent interacts with the environment over a finite horizon of $T$ rounds. At each round $t \in [T]$, the agent selects an action $x_t \in \mathcal{X}$ and observes a noisy evaluation:
\begin{align}\label{eq:regression-setting}
y_t(x) = \pmb\theta_0(x) + \eta_t,
\end{align}
where $\pmb\theta_0 : \mathcal{X} \to \mathbb{R}$ is an unknown objective function and $\{\eta_t\}_{t=1}^T$ are independent Gaussian noise variables with zero mean and known variance. Let $\mathcal{H}_t := \{(x_s, y_s)\}_{s=1}^t$ denote the observation history up to time $t$, and let $\mathcal{F}_t$ be the filtration generated by $\mathcal{H}_t$. We assume that $\pmb\theta_0(\cdot) \in \Theta$, where $\Theta \subseteq L^2(\mathcal{X})$. The space $L^2(\mathcal{X})$ is equipped with the standard inner product
\(
\langle f, g \rangle_{L^2(\mathcal{X})} := \int_{\mathcal{X}} f(x)\, g(x)\, dx.\)

We adopt a Bayesian framework to model uncertainty over the unknown objective function $\pmb\theta_0(\cdot) \in \Theta$ by placing a Gaussian process prior over the function space. A Gaussian process is a collection of random variables $\{\pmb\theta(x)\}_{x \in \mathcal{X}}$ such that, for any finite set $\{x_i\}_{i=1}^m \subset \mathcal{X}$, the random vector $(\pmb\theta(x_1), \ldots, \pmb\theta(x_m))$ follows a multivariate Gaussian distribution with mean function $\mu(\cdot)$ and covariance kernel $k(\cdot,\cdot)$. We consider a zero-mean Gaussian process prior, $\Pi \equiv \mathrm{GP}_{\mathcal{X}}(0, k^\alpha)$, where the scaled kernel is defined as $k^\alpha(\cdot,\cdot) := \alpha^{-1} k(\cdot,\cdot)$, for some $\alpha \in (0,1)$. Observations are assumed to follow the Gaussian noise model, $\eta_t \sim \mathcal{N}_1(0,\lambda)$, with known variance $\lambda > 0$. Given the observation history $\mathcal{D}_t = \{(x_s, y_s)\}_{s=1}^t$, let $A_t := \{x_1,\ldots,x_t\}$ denote the queried points and $y_{1:t} := (y_1,\ldots,y_t)^\top$ the corresponding observations.
Let $\mathcal{F}_t$ denote the sigma algebra generated by $\mathcal{D}_t$.
The resulting $\alpha$-fractional posterior distribution over $\pmb\theta$ is again a Gaussian process:
\[
\Pi_\alpha(\cdot \mid \mathcal{D}_t)
= \mathrm{GP}_{\mathcal{X}}\big(\mu_t(\cdot), k_t^\alpha(\cdot,\cdot)\big),
\quad \text{where } k_t^\alpha(\cdot,\cdot) := \alpha^{-1} k_t(\cdot,\cdot),
\]
with posterior mean and covariance given by:
\begin{align}
\label{eq:fractional-posterior-updates}
\mu_t(x)
&= k(x, A_t)^\top \big(k(A_t, A_t) + \lambda I_t\big)^{-1} y_{1:t}, \\
k_t(x,x')
&= k(x,x') - k(x, A_t)^\top \big(k(A_t, A_t) + \lambda I_t\big)^{-1} k(A_t, x'), \\
\sigma_t^2(x)
&= k_t^\alpha(x,x).
\end{align}
Here, for vectors $E = (e_1,\ldots,e_m)^\top$ and $E' = (e'_1,\ldots,e'_r)^\top$, the matrix $k(E,E') \in \mathbb{R}^{m \times r}$ is defined entrywise by $[k(E,E')]_{i,j} = k(e_i,e'_j)$, and for any function $f \in L^2(\mathcal{X})$ we write $f(E) := (f(e_1),\ldots,f(e_m))^\top$. We assume that the kernel $k$ is continuous, positive definite, and satisfies the integrability condition\(
\int_{\mathcal{X}} \int_{\mathcal{X}} k^2(x,x') \, dx \, dx' < \infty.\)

Given the fractional GP posterior at time $t-1$, we employ a Thompson Sampling strategy to select evaluation points sequentially. At each round $t$, the algorithm samples a function
\(
\pmb\theta_t \sim \Pi_{\alpha}(\cdot \mid \mathcal{D}_{t-1}),
\)
and selects the next query point as
\(
x_t \in \arg\max_{x \in \mathcal{X}} \pmb\theta_t(x).
\)
The selected point $x_t$ is then evaluated to obtain the noisy observation $y_t$, and the posterior is updated accordingly.
This procedure coincides with the GP-TS algorithm in~\cite{Ray&Gopalan2017}. The use of the fractional posterior induces a principled variance inflation while preserving the Thompson Sampling structure. For completeness, the resulting procedure is summarized in Algorithm~\ref{alg:GTS}.
\begin{algorithm}
\caption{Gaussian Process Thompson Sampling (GP-TS) with variance inflation}
\label{alg:GTS}
\begin{algorithmic}
\State \textbf{Input:} Gaussian process prior kernel $k(\cdot,\cdot)$, noise variance $\lambda$, and tempering parameter $\alpha$.
\State Initialize history $\mathcal{D}_0 \gets \emptyset$
\State \textbf{For} $t \in [T]$ \textbf{do}
\State \hspace{1em} Sample $\pmb\theta_t \sim \mathrm{GP}_{\mathcal X}\big(\mu_{t-1}(\cdot), \alpha^{-1}k_{t-1}(\cdot,\cdot)\big)$
\State \hspace{1em} Select $x_t \gets \arg\max_{x \in \mathcal X} \pmb \theta_t(x)$
\State \hspace{1em} Observe $y_t \gets \pmb\theta_0(x_t) + \eta_t$
\State \hspace{1em} Update $\mathcal{D}_t \gets \mathcal{D}_{t-1} \cup \{(x_t,y_t)\}$
\State \hspace{1em} Update posterior $\mu_t(\cdot),\, k_t(\cdot,\cdot)$
\State \textbf{end for}
\end{algorithmic}
\end{algorithm}

From Algorithm~\ref{alg:GTS}, the joint data-generating distribution induced by GP-TS up to time $T$ is given by:
\[
\mathbb{P}_0^{(T)}(D_T)
:=
\prod_{t=1}^{T}
\bigl[
\mathcal{N}_1(y_t \mid \pmb\theta_0(x_t),\lambda)\,
q_{\alpha}(x_t \mid \mathcal{D}_{t-1})
\bigr],
\]
where $q_{\alpha}(\cdot \mid \mathcal{D}_{t-1})$ denotes the distribution over query points induced by posterior sampling from the $\alpha$-fractional GP posterior at time $t$. 

For any function $\pmb\theta \in \Theta$ and any sequence of query points
$A_t$, under the Gaussian noise model, the conditional distribution of $y_{1:T}$
given $A_t$ is
\(
y_{1:T} \mid A_T, \pmb\theta
\;\sim\;
\mathcal{N}_T\!\left(
\big(\pmb\theta(x_1),\ldots,\pmb\theta(x_T)\big),
\lambda I_T
\right),
\)
or equivalently,
\(
p_{\pmb\theta}^{(T)}(y_{1:T} \mid A_T)
=
\prod_{t=1}^{T} \mathcal{N}_1\!\left(y_t \mid \pmb\theta(x_t), \lambda\right).
\)
We denote the associated probability measure by $\mathbb{P}_{\pmb\theta,x}^{(T)}$. To quantify separation between reward-generating distributions, for any
$\beta > 0$ and $\pmb\theta,\pmb\vartheta \in \Theta$, we define the
order-$\beta$ R\'enyi divergence at time $t$ as,
\(
D_{\beta}^{(t)}(\pmb\theta,\pmb\vartheta)
:=
\frac{1}{\beta-1}
\log
\int
p_{\pmb\theta}^{(t)}(y_{1:t} \mid A_t)^{\beta}
p_{\pmb\vartheta}^{(t)}(y_{1:t} \mid A_t)^{1-\beta}
\, dy_{1:t}
\). 
\paragraph{Regret.} Our objective is to analyze the frequentist performance of GP-TS under fractional posterior inference. Performance is quantified using the standard notion of cumulative (pseudo-)regret, which measures the total loss incurred by the algorithm from failing to select the globally optimal decision over a finite horizon of $T$ rounds. Formally, the cumulative regret of GP-TS with fractional posterior parameter $\alpha$ is defined as:
\begin{align}
\label{eq:regret-def}
\mathrm{Regret}(T)
:= \sum_{t=1}^{T} \bigl[ \pmb\theta_0(x_0) - \pmb\theta_0(x_t) \bigr],
\end{align}
where $x_0 \in \arg\max_{x \in \mathcal{X}} \pmb\theta_0(x)$ denotes a global maximizer of the unknown objective function. The randomness in \eqref{eq:regret-def} arises from both the observation noise and the internal randomization of the GP-TS policy. Our main results establish upper bounds on the expected regret, $\mathbb{E}[\mathrm{Regret}(T)]$, under appropriate assumptions on the Gaussian process prior and the tempering parameter $\alpha$. Throughout, expectations of the cumulative regret are taken with respect to $\mathbb{P}_0^{(T)}$.

\section{Frequentist Regret Bound}\label{sec:reganal}

As discussed in Section~\ref{sec:introduction}, our regret analysis is driven by first-order concentration properties of the $\alpha$-fractional posterior distribution. We begin by imposing a joint regularity condition on the prior and the data-generating process, which ensures near-optimal contraction of the $\alpha$-posterior and forms the basis of our regret bounds.

\begin{assumption}[Prior thickness]\label{ass:prior}
Fix $\alpha \in (0,1)$ and let $\{\epsilon_t\}_{t\geq 0}$ be a positive sequence. There exists $t_0 \geq 1$ such that for all $t \geq t_0$, 
\(
\Pi\!\left( B(\pmb\theta_0, \epsilon_t) \right)
\;\geq\;
\exp\!\left(-\frac{D \alpha t \epsilon_t^2}{4}\right),
\)
where
\(
B(\pmb\theta_0, \epsilon_t)
:=
\left\{
\pmb\theta \in \Theta :
D_2^{(t)}(\pmb\theta_0, \pmb\theta)
\leq \frac{D \alpha t \epsilon_t^2}{4}
\right\}
\)
denotes a neighborhood of $\pmb\theta_0(\cdot)$ defined with respect to the query points $\{x_s\}_{s=1}^t$, and $D>0$ is a fixed constant.
\end{assumption}

Assumption~\ref{ass:prior} is standard in Bayesian nonparametrics and ensures that the prior places sufficient mass in shrinking neighborhoods of the true function to achieve the posterior contraction rate $\epsilon_t$. When this contraction rate is coupled with an upper bound on the cumulative posterior variance,
\(\mathbb{E}\!\left[\sum_{t=1}^{T} k_{t-1}^2(x_t,x_t)\right], \)
the two ingredients together imply frequentist regret guarantees via Theorem~\ref{thm:GRB}. 

We now state our main regret bound. The result below provides a general upper bound on the expected cumulative regret of GP-TS under fractional posterior inference with $\alpha t\epsilon_t^2 <1$.

\begin{theorem}[General regret bound]\label{thm:GRB}
Under Assumption~\ref{ass:prior}, for any $\eta_1, \eta_2 \in (0,1)$ and $\alpha t\epsilon_t^2 <1$, the expected cumulative regret satisfies:
\begin{align}
\label{eq:regret-bound-final-alpha-theorem}
\mathbb{E}[\mathrm{Regret}(T)]
&\leq
\Bigg\{
\big[\overline{p}_0(1-\eta_1^{-2}-\eta_2^{-1})\big]^{-\frac{1}{2}}
\Bigg[
\left(\frac{D}{1-\alpha}\right)^{\frac{1}{2}}
+
\left(\frac{(D+2)\,C(1-\alpha,\pmb\theta_0,k)}{1-\alpha}\right)^{\frac{1}{2}}
\Bigg] \nonumber\\
&\quad\quad
+
\left(\frac{(D+2)\,C(1-\alpha,\pmb\theta_0,k)}{1-\alpha}\right)^{\frac{1}{2}}
\Bigg\}
\left(
\mathbb{E}\!\left[\sum_{t=1}^T k_{t-1}^2(x_t,x_t)\right]
\right)^{\frac{1}{2}}
\left(
\sum_{t=1}^T t \epsilon_t^2
\right)^{\frac{1}{2}},
\end{align}
where $\overline{p}_0$ is a positive constant and $C(\cdot)$ is a known function depending on the prior, kernel $k(\cdot, \cdot)$, and $\pmb\theta_0(\cdot)$.
\end{theorem}

Our proof adapts the saturated–unsaturated action decomposition originally introduced by~\cite{agrawal2013thompsonLIN}, combined with expectation-based concentration properties of the $\alpha$-posterior.

\begin{definition}[Saturated decisions]\label{def:US}
At time $t$, a decision point $x \in \mathcal{X}$ is said to be \emph{saturated} if
\(
\pmb\theta_0(x_0) - \pmb\theta_0(x)
\;\geq\;
2\mathcal{C}_t \sigma_{t-1}(x),
\)
where $\sigma_{t-1}(x)$ denotes the posterior standard deviation and
\(
\mathcal{C}_t^2 := \frac{D \alpha t \epsilon_t^2}{1-\alpha}.
\)
The set of saturated decisions at time $t$ is denoted by $\mathcal{S}_t$.
\end{definition}

Definition~\ref{def:US} mirrors the construction in~\cite{agrawal2013thompsonLIN}, with the key distinction that the confidence threshold $\mathcal{C}_t$ is determined by the contraction rate $\epsilon_t$ of the $\alpha$-GP posterior. Notably, the optimal action $x_0$ is always unsaturated.

A proof sketch of Theorem~\ref{thm:GRB}, highlighting the role of first-order posterior concentration, is provided below, while the full proof is deferred to Appendix~\ref{app:thm-GRB}. In Section~\ref{subsec:first-order-concentration}, we establish posterior contraction rates for the $\alpha$-posterior, and in Section~\ref{subsec:bound-common-kernels} we specialize Theorem~\ref{thm:GRB} to commonly used kernels to obtain explicit regret rates.

\paragraph{Proof sketch of Theorem~\ref{thm:GRB}.}
The proof proceeds by decomposing the instantaneous regret into components that can be controlled using posterior uncertainty and posterior contraction properties of the $\alpha$-fractional GP posterior, and then aggregating these bounds over time.

At each round $t$, we introduce an auxiliary decision point $\overline{x}_t := \arg\min_{x \notin \mathcal{C}_t} \sigma_{t-1}^2(x)$, where $\mathcal{C}_t$ denotes the set of saturated decisions. Using this point, the instantaneous regret admits the decomposition:
\begin{align}
\pmb\theta_0(x_0) - \pmb\theta_0(x_t)
=
\bigl[\pmb\theta_0(x_0) - \pmb\theta_0(\overline{x}_t)\bigr]
+
\bigl[\pmb\theta_0(\overline{x}_t) - \pmb\theta_t(x_t)\bigr]
+
\bigl[\pmb\theta_t(x_t) - \pmb\theta_0(x_t)\bigr].
\label{eq:ps1}
\end{align}
This decomposition separates the regret into an approximation error associated with unsaturated decisions, a sampling error arising from maximizing a posterior sample, and an estimation error due to posterior uncertainty. By Definition~\ref{def:US}, unsaturated decisions satisfy $\pmb\theta_0(x_0)-\pmb\theta_0(\overline{x}_t)\le 2\mathcal{C}_t\sigma_{t-1}(\overline{x}_t)$. Applying the Cauchy--Schwarz inequality over the cumulative expectation of the first term in~\eqref{eq:ps1} yields a bound involving $\sum_{t=1}^T \mathcal{C}_t^2$ and $\sum_{t=1}^T \mathbb{E}[\sigma_{t-1}^2(\overline{x}_t)]$:
\[
\sum_{t=1}^T \mathbb{E}\!\left[\pmb{\theta}_0(x_0) - \pmb{\theta}_0(\overline{x}_t)\right]
\lesssim
\Big(\sum_{t=1}^T \mathcal{C}_t^2\Big)^{\frac{1}{2}}
\Big(\sum_{t=1}^T \mathbb{E}[\sigma_{t-1}^2(\overline{x}_t)]\Big)^{\frac{1}{2}}.
\]
To control the estimation error term (last term in~\eqref{eq:ps1}), we add and subtract the posterior mean to write,
$\pmb\theta_t(x_t)-\pmb\theta_0(x_t)
=
\bigl[\pmb\theta_t(x_t)-\mu_{t-1}(x_t)\bigr]
+
\bigl[\mu_{t-1}(x_t)-\pmb\theta_0(x_t)\bigr]$.
Note that, the first term has zero expectation. For the second term, the reproducing property of the kernel $k_{t-1}^\alpha(\cdot, \cdot)$ implies
$
|\mu_{t-1}(x_t)-\pmb\theta_0(x_t)|
\le
\sigma_{t-1}(x_t)\|\mu_{t-1}-\pmb\theta_0\|_{k_{t-1}^\alpha}.
$
Aggregating this bound over time and applying Cauchy--Schwarz inequality shows that the cumulative estimation error is controlled by the product of the cumulative posterior variance along the played actions, $\sum_{t=1}^T \sigma_{t-1}^2(x_t)$, and the cumulative squared RKHS error, $\sum_{t=1}^T \|\mu_{t-1}-\pmb\theta_0\|_{k_{t-1}^\alpha}^2$, yielding:
\[
\sum_{t=1}^T \mathbb{E}\!\left[\pmb{\theta}_t(x_t) - \pmb{\theta}_0(x_t)\right]
\le
\Big(\mathbb{E}\!\sum_{t=1}^T\sigma_{t-1}^2(x_t)\Big)^{\!\frac{1}{2}}
\Big(\sum_{t=1}^T\mathbb{E}\!\left[\|\mu_{t-1}-\pmb{\theta}_0\|_{k_{t-1}^{\alpha}}^2\right]\Big)^{\!\frac{1}{2}}.
\]
Since $x_t$ maximizes $\pmb{\theta}_t(x)$, the same argument that we used to bound the third term applies with $\overline{x}_t$ in place of $x_t$ to bound the second term in \eqref{eq:ps1} as:
\[
\sum_{t=1}^T \mathbb{E}\!\left[\pmb{\theta}_0(\overline{x}_t)-\pmb{\theta}_t(x_t)\right]
\le
\Big(\mathbb{E}\!\sum_{t=1}^T\sigma_{t-1}^2(\overline{x}_t)\Big)^{\!\frac{1}{2}}
\Big(\sum_{t=1}^T\mathbb{E}\!\left[\|\mu_{t-1}-\pmb{\theta}_0\|_{k_{t-1}^{\alpha}}^2\right]\Big)^{\!\frac{1}{2}}.
\]

\medskip
\noindent\textit{Posterior contraction via prior thickness}.
The prior thickness condition in Assumption~\ref{ass:prior}, together with properties of the $\alpha$-fractional posterior and R\'enyi divergence, implies a first-order posterior contraction bound. In particular, Lemma~\ref{lem:post} shows that:
\[
\sum_{t=1}^T \mathbb{E}\!\left[\|\mu_{t-1}-\pmb\theta_0\|_{k_{t-1}^\alpha}^2\right]
\;\lesssim\;\alpha
\sum_{t=1}^T t\epsilon_t^2.
\]

This step translates statistical concentration of the posterior around the true function into a quantitative control on cumulative estimation error, which is a key ingredient in the regret bound.

\medskip
\noindent\textit{Lower bounding the probability of playing an unsaturated decision point}.
To control $\sum_{t=1}^T \mathbb{E}[\sigma_{t-1}^2(\overline{x}_t)]$, we relate $\sigma_{t-1}^2(\overline{x}_t)$ to $\sigma_{t-1}^2(x_t)$ using the fact that GP-TS selects an \emph{unsaturated} decision point with non-negligible probability. Recall that $\overline{x}_t$ is the minimum-variance point among unsaturated decisions and is $\mathcal{F}_{t-1}$-measurable. Conditioning on $\mathcal{F}_{t-1}$, we first note that
$\mathbb{E}[\sigma_{t-1}^2(x_t)\mid\mathcal{F}_{t-1}]
\ge \sigma_{t-1}^2(\overline{x}_t)\,\mathbb{P}(x_t\notin\mathcal{C}_t\mid\mathcal{F}_{t-1})$,
so it suffices to lower bound $\mathbb{P}(x_t\notin\mathcal{C}_t\mid\mathcal{F}_{t-1})$.

Since the optimal decision point $x_0$ is always unsaturated, the event $\{\pmb\theta_t(x_0)>\pmb\theta_t(x)\ \forall x\in\mathcal{C}_t\}$ implies that an unsaturated decision point is selected. We therefore lower bound
$\mathbb{P}(x_t\notin\mathcal{C}_t\mid\mathcal{F}_{t-1})$
by the probability that the posterior sample ranks $x_0$ above all saturated decisions. To obtain a tractable bound, we intersect this ranking event with a second event ensuring that the posterior sample is \emph{uniformly close} to the truth up to the posterior standard deviation. Concretely, define
$\mathcal{E}_t := \left\{\sup_{x\in\mathcal{X}} \frac{\pmb\theta_t(x)-\pmb\theta_0(x)}{\sigma_{t-1}(x)} \le 2\mathcal{C}_t\right\}$.
On $\mathcal{E}_t$, every saturated decision point $x\in\mathcal{C}_t$ satisfies
$\pmb\theta_t(x) \le \pmb\theta_0(x)+2\mathcal{C}_t\sigma_{t-1}(x)
\le \pmb\theta_0(x_0)$
by the definition of saturation. Hence, if in addition $\pmb\theta_t(x_0)>\pmb\theta_0(x_0)$ holds, then necessarily $\pmb\theta_t(x_0)>\pmb\theta_t(x)$ for all $x\in\mathcal{C}_t$, implying $x_t\notin\mathcal{C}_t$. This yields the key lower bound:
\[
\mathbb{P}(x_t\notin\mathcal{C}_t\mid\mathcal{F}_{t-1})
\;\ge\;
\mathbb{P}\big(\pmb\theta_t(x_0)>\pmb\theta_0(x_0)\mid\mathcal{F}_{t-1}\big)
-
\mathbb{P}\big(\mathcal{E}_t^{c}\mid\mathcal{F}_{t-1}\big),
\]
where the subtraction comes from a union-bound argument.

The first term is an anti-concentration event for a one-dimensional Gaussian: conditional on $\mathcal{F}_{t-1}$, $\pmb\theta_t(x_0)$ is Gaussian with mean $\mu_{t-1}(x_0)$, and variance $\sigma_{t-1}^2(x_0)$, so $\mathbb{P}(\pmb\theta_t(x_0)>\pmb\theta_0(x_0)\mid\mathcal{F}_{t-1})$ can be lower bounded in terms of the standardized gap, $(\pmb\theta_0(x_0)-\mu_{t-1}(x_0))/\sigma_{t-1}(x_0)$ (Lemma~\ref{lem:LB}). The second term is controlled by splitting
$\pmb\theta_t-\pmb\theta_0 = (\pmb\theta_t-\mu_{t-1})+(\mu_{t-1}-\pmb\theta_0)$.
The second component is handled via expectation-based posterior contraction bounds, yielding
$\mathbb{P}(\|\mu_{t-1}-\pmb\theta_0\|_{k_{t-1}^{\alpha}}\ge \eta_1\mathcal{C}_t)\le \eta_1^{-2}$
(for suitable $\eta_1$), while the first component is a centered GP sample so Gaussian tail bounds control
$\sup_x (\pmb\theta_t(x)-\mu_{t-1}(x))/\sigma_{t-1}(x)$.
Combining these arguments gives a lower bound of the form
$\mathbb{P}(x_t\notin\mathcal{C}_t\mid\mathcal{F}_{t-1})\ge \overline{p}_0$,
where $\overline{p}_0$ is explicit and strictly positive when $\alpha t\epsilon_t^2<1$. This yields the desired comparison,
$\mathbb{E}\big[\sum_{t=1}^T \sigma_{t-1}^2(\overline{x}_t)\big]
\le \overline{p}_0^{-1}\,\mathbb{E}\big[\sum_{t=1}^T \sigma_{t-1}^2(x_t)\big]$
(up to the high-probability event quantified in the detailed proof).

\medskip
\noindent\textit{Reduction to cumulative posterior variance and aggregation}.
The above probability bound allows us to relate the cumulative variance along unsaturated decisions to the variance along the played actions, yielding
$
\mathbb{E}\!\left[\sum_{t=1}^T \sigma_{t-1}^2(\overline{x}_t)\right]
\le
\overline{p}_0^{-1}
\mathbb{E}\!\left[\sum_{t=1}^T \sigma_{t-1}^2(x_t)\right].
$
Combining this bound with the previous estimates yields:
\[
\mathbb{E}[\mathrm{Regret}(T)]
\;\lesssim\;
\left(\mathbb{E}\!\left[\sum_{t=1}^T k_{t-1}(x_t,x_t)\right]\right)^{\frac{1}{2}}
\left(\sum_{t=1}^T t\epsilon_t^2\right)^{\frac{1}{2}},
\]
up to constants depending on the $\pmb\theta_0(\cdot)$, noise variance, and $\alpha$.

\section{Regret Bounds for Common Kernels}\label{subsec:bound-common-kernels}
Here, we present the bounds of $\mathbb{E}[\mathrm{Regret}(T)]$ based on Theorem~\ref{thm:GRB} for commonly used covariance kernels, such as the squared exponential, Mat\'{e}rn-$\nu$, and rational quadratic kernels respectively.
Using~\cite{chowdhury2021reinforcement}, we have $\mathbb{E}\left[\sum_{t=1}^{T}k_{t-1}(x_t, x_t)\right] \leq 4\gamma_T$, where the maximum information gain after $T$ rounds is $\gamma_T := \max_{A\subset \mathcal{X}:|A|=T}I(y_A; f_A)$.
It has been established by~\cite{srinivas2010gaussian}[Theorem 5] that, for a compact and convex $\mathcal{X}\subset \mathbb{R}^{d}$ with $d\in \mathbb{N}$, $\gamma_T$ is $\mathcal{O}(\log^{d+1} T)$ and $\mathcal{O}(T^{\frac{d(d+1)}{2\nu+d(d+1)}}\log T)$ for the squared exponential and Mat\'{e}rn-$\nu$ kernels respectively. For the Mat\'{e}rn-$\nu$ kernel, a tighter bound for $\gamma_T$ was later provided by~\cite{vakili21a}, where the polynomial dependence on $T$ was reduced to logarithmic factor, yielding $\mathcal{O}(T^{\frac{d}{2\nu+d}}\log^{\frac{2\nu}{2\nu+d}} T)$. 

\paragraph{Squared Exponential kernel.} For the squared exponential kernel, $k_a(x, x') = e^{-a^2\lVert x-x'\rVert^2_2}$ for $a>0$, and the true function lying in the prior RKHS, i.e., $\pmb{\theta}_0(\cdot)  \in \mathcal{H}_0$, we show in Lemma~\ref{lemma:SE-contraction-rate} of Appendix~\ref{app:SE-contraction} that, under suitable assumptions on the prior for $a$, the $\alpha$-posterior contracts at the minimax rate (up to logarithmic terms) of $\epsilon_t \lesssim t^{-\frac{1}{2}}\log^{\frac{d+1}{2}} t$. In particular, one obtains:
\begin{align}\label{eq:e_t_bound_SE}
    \begin{split}
        \sum_{t=1}^{T}t\epsilon_t^{2} &\lesssim \sum_{t=1}^{T} \log^{d+1} t \leq T \log^{d+1}T.
    \end{split}
\end{align}
Therefore, combining this observation in \eqref{eq:e_t_bound_SE} with $\gamma_T$ for squared exponential kernel from~\cite{srinivas2010gaussian,vakili21a} and Theorem \ref{thm:GRB}, we have:
\begin{align}\label{eq:SE-RB}
    \mathbb{E}[\mathrm{Regret}(T)] = \mathcal{O}\left(T^{\frac{1}{2}} \log^{d+1} T\right).
\end{align}

\paragraph{Mat\'{e}rn-$\nu$ kernel.} In case of the Mat\'{e}rn-$\nu$ kernel, $k^{(\nu)}_a(x, x') = 2^{1-\nu}(\Gamma(\nu))^{-1}(\sqrt{2\nu}a\lVert x-x'\rVert_2)^{\nu}K_{\nu}$ $(\sqrt{2\nu}a\lVert x-x'\rVert_2)$ for $a, \nu > 0$,
Lemma~\ref{lemma:matern-contraction-rate} yields the minimax contraction rate of the $\alpha$-posterior for this kernel as $\epsilon_t = t^{-\frac{\nu}{2\nu+d}} \log^{\frac{q}{2+d}} t$, where $q$ denotes the prior parameter of $a$. Hence, in a similar fashion, one can bound $\sum_{t=1}^{T}t\epsilon_t^2$ as:
\begin{align}\label{eq:bound-t-e_t}
    \sum_{t=1}^{T}t\epsilon_t^2 = \sum_{t=1}^{T} t^{1-\frac{2\nu}{2\nu+d}}\log^{\frac{2q}{2+d}} t \leq  \log^{\frac{2q}{2+d}
    } T\int_{0}^{T}t^{1-\frac{2\nu}{2\nu+d}}dt
        = \frac{T^{2-\frac{2\nu}{2\nu+d}}}{2-\frac{2\nu}{2\nu+d}}\log^{\frac{2q}{2+d}} T.
\end{align}
Combining \eqref{eq:bound-t-e_t} with $\gamma_T$ for Mat\'{e}rn-$\nu$ kernel from~\cite{vakili21a} and Theorem \ref{thm:GRB}, we have:
\begin{align}\label{eq:Matern-RB}
    \mathbb{E}[\mathrm{Regret}(T)] = \mathcal{O}\left(T^{\frac{2\nu+3d}{2(2\nu+d)}} \log^{\frac{\nu}{2\nu+d} + \frac{q}{2+d}} T\right).
\end{align}
Both the regret bounds for the squared exponential and Mat\'{e}rn-$\nu$ kernels in \eqref{eq:SE-RB} and \eqref{eq:Matern-RB}, respectively, are derived under variable smoothness levels. While the bound for the squared exponential kernel matches the rate obtained in~\cite{Ray&Gopalan2017} for infinitely smooth function classes, the bound for the Mat\'{e}rn-$\nu$ kernel is looser than the state-of-the-art by a factor of $T^{\frac{d}{2(2\nu+d)}}$.

\paragraph{Rational Quadratic kernel.} 
For the rational quadratic kernel,
\(
k_{\mathrm{RQ}}(x,x')
=
\left(1 + \frac{\|x-x'\|_2^2}{2\nu \ell^2}\right)^{-\nu}
\)
with $\nu, \ell > 0$,
Lemma~\ref{lemma:rq-contraction-rate} characterizes the contraction rate of the corresponding $\alpha$-posterior as $\epsilon_t \asymp  t^{-\frac{\nu}{2\nu+d}}$. Consequently, one can bound the cumulative contraction term $\sum_{t=1}^{T} t \epsilon_t^2$ in a similar way as in~\eqref{eq:bound-t-e_t}. 
Combining this with the bound on the maximum information gain for the rational quadratic kernel from Lemma~\ref{lem:rq-gamma} that is \(\gamma_T
=
\mathcal{\tilde O} (T^{\frac{d}{2\nu+d}})\), and applying Theorem~\ref{thm:GRB}, we obtain the following regret bound:
\begin{align}\label{eq:RQ-RB}
\mathbb{E}[\mathrm{Regret}(T)]
=
\mathcal{\tilde O}\!\left(
T^{\frac{2\nu+3d}{2(2\nu+d)}}
\right).
\end{align}
Similar to the squared exponential and Mat\'ern-$\nu$ cases, the regret bound for the rational quadratic kernel in~\eqref{eq:RQ-RB} is derived under a variable smoothness regime. The resulting rate reflects the polynomial eigenvalue decay of the rational quadratic kernel and follows directly from our general regret bound by substituting kernel-specific expressions for $\gamma_T$ and $\epsilon_T$.

\paragraph{Variance inflation in GP-TS.} The analysis of~\cite{Ray&Gopalan2017} introduces an explicit variance inflation factor $\gamma_T$ in GP-TS in order to ensure sufficient exploration and to simplify the regret analysis. In our framework, variance inflation arises implicitly through the fractional posterior parameter $\alpha$, with the effective inflation factor given by $\alpha^{-1}$. Assuming $T$ is known in advance, our regret bound holds for  $\alpha^{-1}= T\epsilon_T^2\; (> t\epsilon_t^2)$, where $\epsilon_T$ denotes the posterior contraction rate at time $T$. Instantiating this condition for commonly used kernels yields inflation factors that match those in~\cite{Ray&Gopalan2017}, up to logarithmic terms. In particular, the contraction rate for the squared exponential kernel implies $\alpha^{-1} = \mathcal{\tilde{O}}(\log^{d+1} T)$. For the Mat\'ern-$\nu$ kernel, $\epsilon_T^2 = \mathcal{\tilde{O}}(T^{-\frac{2\nu}{2\nu+d}})$, leading to $\alpha^{-1} = \mathcal{\tilde{O}}(T^{\frac{d}{2\nu+d}})$. These inflation factors coincide, up to logarithmic factors, with the information gain term $\gamma_T$ used in~\cite{Ray&Gopalan2017}. 

This comparison suggests that the variance inflation used in existing analyses of GP-TS can be interpreted within our framework as arising from posterior contraction properties, rather than being introduced solely as an algorithmic choice. From this perspective, fractional posteriors offer a statistically motivated way to parameterize variance inflation in GP-TS. At the same time, our analysis highlights an open question regarding the relationship between the information gain parameter $\gamma_T$ and the posterior contraction rate $\epsilon_T^2/T$.

\section{\texorpdfstring{First-order Concentration Properties of the $\alpha$-Posterior Distribution}{First-order Concentration Properties of the alpha-Posterior Distribution}}\label{subsec:first-order-concentration}
In this section, we provide a more general result characterizing the first-order concentration properties of any nonparametric $\alpha$-posterior distribution supported on $\Theta \subseteq L^2(\mathcal X)$, induced by a Gaussian process prior, denoted as $\Pi(\cdot)$ throughout, with reproducing kernel $k(\cdot, \cdot)$. The general observation model has density $p_{\pmb\theta}(y\mid x)$, with the true data-generating distribution given by $p_{0}(y \mid x)$ corresponding to $\pmb \theta=\pmb \theta_0$. For any measurable set $A\in \Theta$, we define the $\alpha$-posterior distribution (for any $\alpha\in(0,1)$) as: 
\begin{align}\label{eq:alpha-posterior}
    \Pi_{t,\alpha}(A \mid \mathcal{F}_t)&= \frac{ \int_{A}  \prod_{s=1}^{t} [p_{\pmb \theta}(y_s \mid x_s)]^{\alpha}\Pi(d\pmb \theta) } { \int_\Theta  \prod_{s=1}^{t} [p_{\pmb \theta}(y_s \mid x_s)]^{\alpha} \Pi(d\pmb \theta)  }.
\end{align}
In the preceding definition given in \eqref{eq:alpha-posterior}, the policy that selects the decision point point $x_s$ at iteration $s$ does not appear in either the numerator or the denominator because it cancels out, as it is independent of $\pmb{\theta}$. We now establish a finite-time concentration result for the $\alpha$-fractional GP posterior, which serves as the main statistical ingredient in our regret analysis. In contrast to standard posterior contraction results that are typically stated asymptotically or in probability, our analysis requires both high-probability and expectation-based bounds that hold uniformly over the sequentially collected data. Leveraging Assumption~\ref{ass:prior}, we derive such bounds as follows.

\begin{lemma}[$\alpha$-posterior contraction]\label{lem:post}
Fix $\alpha\in(0,1)$ and $\lambda>0$. Suppose the Gaussian process prior satisfies the integrability condition
$\int \Pi(d\pmb \theta)\exp\{(1-\alpha)\|\pmb \theta\|_k^2\}<\infty$,
and that Assumption~\ref{ass:prior} holds for all $t>t_0$. Then, for any $j>0$, the following statements hold:
\begin{align}\label{eq:post-1}
\mathbb{P}_{0}\!\left(
\Pi_{t,\alpha}\!\left(
\frac{1}{2}\|\pmb \theta-\pmb \theta_0\|_{k^\alpha}^2
+
D_{\alpha}^{(t)}(\pmb \theta,\pmb \theta_0)
\ge
\frac{(D+j)\alpha}{2(1-\alpha)}\,t\epsilon_t^2
\;\middle|\;
\mathcal{F}_{t}
\right)
\Big|\,
A_t
\right)
\le
2C(1-\alpha,\pmb\theta_0,k)
e^{-\frac{j\alpha}{4}t\epsilon_t^2},
\end{align}
and:
\begin{align}\label{eq:post-2}
\mathbb{E}\!\left[
\mathbb{E}\!\left[
\frac{1}{2}\|\pmb \theta-\pmb \theta_0\|_{k^\alpha}^2
+
D_{\alpha}^{(t)}(\pmb \theta,\pmb \theta_0)
\;\middle|\;
\mathcal{F}_{t}
\right]
\Big|\,
A_t
\right]
\le
\frac{(D+2)\alpha}{2(1-\alpha)}\,t\epsilon_t^2
+
\frac{4C(1-\alpha,\pmb\theta_0,k)
e^{-\frac{\alpha}{4}t\epsilon_t^2}}{1-\alpha},
\end{align}
where $C(1-\alpha,\pmb\theta_0,k):= e^{{ (1-\alpha)} \|\pmb \theta_0\|_k^2} \int \Pi(d\pmb \theta)e^{{ (1-\alpha)} \|\pmb \theta\|_k^2}$ is a known constant depending only on the prior, kernel $k(\cdot, \cdot)$, and $\pmb\theta_0(\cdot)$.
\end{lemma}

The proof of Lemma~\ref{lem:post} is given in Appendix~\ref{app:Post}. The first inequality in \eqref{eq:post-1} provides a high-probability bound on the $\alpha$-posterior mass assigned to functions that are far from the truth, thereby characterizing the finite-sample contraction behavior of the posterior. In particular, it shows that the posterior probability of large deviations decays exponentially at rate $t\epsilon_t^2$. The second inequality in~\eqref{eq:post-2} establishes a complementary concentration bound in expectation. This bound plays a crucial role in our regret analysis, as it allows us to control the cumulative estimation error of the posterior mean without constructing high-probability confidence sets or discretizing the domain. Intuitively, since the R\'enyi divergence term $D_{\alpha}^{(t)}(\pmb \theta,\pmb \theta_0)$ scales linearly with the number of observations $t$, the sequence $\epsilon_t$ governs the rate at which the $\alpha$-posterior concentrates around the true function.

Both results are adapted from posterior concentration techniques available in the Bayesian statistics literature on fractional posteriors, but are stated here in a form tailored to sequential decision-making. In particular, the expectation-based bound in \eqref{eq:post-2} is essential for deriving regret guarantees in continuous action spaces. We further discuss, in Remark~\ref{remark:diff-standard-regret}, a subtle yet important distinction between our analysis and standard $\alpha$-posterior concentration results that underpins the subsequent regret bounds.

\begin{remark}
\label{remark:diff-standard-regret}
Typical analysis of the $\alpha$-posterior concentration bounds measure deviation using only the $D_{\alpha}^{(t)}(\pmb \theta,\pmb \theta_0)$ term. However, in the analysis of the regret, we needed to quantify a bound on the expectation of $\|\pmb \theta-\pmb \theta_0\|_{k_{t}}$, where $k_{t}(\cdot, \cdot)$ is the kernel of the GP posterior at time $t$. Using the identity derived in~Lemma~\ref{lem:identity}, we express $\|\pmb \theta-\pmb \theta_0\|_{k_{t}}=\|\pmb \theta-\pmb \theta_0\|_{k} + \alpha \lambda^{-1}\sum_{s=1}^{t} (\theta(x_s)-\theta_0(x_s))^2$. Note that for Gaussian noise with variance $\lambda$,  $2 D_{\alpha}^{(t)}(\pmb \theta,\pmb \theta_0)= \alpha \lambda^{-1}\sum_{s=1}^{t} (\theta(x_s)-\theta_0(x_s))^2$. Accordingly, we modify the analysis to incorporate this additional term $\|\pmb \theta-\pmb \theta_0\|_{k^\alpha}$. Also, it is only due to this modified definition of the discrepancy measure, we needed an additional assumption (highlighted in the statement of Lemma~\ref{lem:post}) on the prior that requires its tails to be sub-Gaussian. 
\end{remark}

\section{Conclusion}\label{sec:conclusion}
In this paper, we study GP-TS through the lens of fractional posteriors and show that the commonly used variance inflation in GP-TS can be formalized in a principled manner via an $\alpha$-fractional GP posterior. This perspective provides a statistical interpretation of variance inflation as a mechanism that balances posterior concentration and exploration, rather than as an ad hoc analytical device.

By leveraging tools from the Bayesian statistics literature on fractional posteriors~\citep{bhattacharya2019bayesian}, we develop an expectation-based analysis that avoids domain discretization and yields finite-time, near-optimal frequentist regret bounds. Our results apply to the widely used squared exponential and Mat\'{e}rn-$\nu$ kernels, and match existing lower bounds up to logarithmic factors, thereby resolving a key limitation in prior analyses of GP-TS.

Beyond the specific regret guarantees obtained here, our analysis introduces a modular framework that decouples posterior contraction properties from kernel-dependent variance growth. This structure suggests that regret bounds for GP-TS can be derived systematically for a broader class of kernels, provided corresponding posterior contraction rates are available.

A natural direction for future work is to extend this framework to Bayesian Optimization algorithms with richer model structures, such as Gaussian process priors composed with activation functions or neural network feature maps. Understanding how fractional posteriors interact with such nonlinear representations will potentially provide a principled path toward regret guarantees for modern Bayesian Optimization methods used in high-dimensional and nonparametric settings.

\bibliography{references}

\appendix


\section{Reproducing Kernel Hilbert Space (RKHS) and Mercer's Theorem}\label{app:RKHS-Mercer}

We briefly review the definition and key properties of RKHS in Definition \ref{def:RKHS} and Remark \ref{remark:RKHS} below.

\begin{definition}[Reproducing kernel Hilbert space (RKHS)]\label{def:RKHS}
Let $k:\mathcal X \times \mathcal X \to \mathbb{R}$ be a positive definite kernel with respect to a finite Borel measure supported on $\mathcal X$. A Hilbert space $\mathcal H_{k}$ of functions on $\mathcal X$ equipped with the inner product $\langle \cdot, \cdot \rangle_{\mathcal H_k}$ is called reproducing kernel Hilbert space (RKHS) with $k(\cdot, \cdot)$ if it satisfies: (a) $\forall x\in \mathcal X, k(\cdot, x)\in \mathcal H_k$ and (b) $\forall x\in \mathcal X$ and $f\in \mathcal H_k$, $\langle f, k(\cdot, x)\rangle_{\mathcal H_k} = f(x)$ (reproducing or evaluation property).
\end{definition}

\begin{remark}[An RKHS and its reproducing kernel uniquely determine one another]\label{remark:RKHS}
The structure of an RKHS, as in Definition \ref{def:RKHS} above, is completely characterized by its kernel, and conversely, every positive definite kernel defines a unique RKHS~\citep{aronszajn1950theory}.
\end{remark}

We state Mercer's theorem in Theorem \ref{thm:Mercer} below without proof, which provides an alternative representation for GP kernels as an inner product of infinite dimensional feature maps~\citep{kanagawa2018gaussian}[Theorem 4.1].

\begin{theorem}[Mercer's theorem]\label{thm:Mercer}
    Let $k(\cdot, \cdot)$ be a continuous kernel with respect to a finite Borel measure on $\mathcal X$. There exists $\{(\lambda_m, \phi_m)\}_{m=1}^{\infty}$ such that $\lambda_m\in \mathbb{R}^{+}$, $\phi_m\in \mathcal H_k$, for $m\geq 1$, and:
    \begin{align*}
        k(x, x') = \sum_{m=1}^{\infty}\lambda_m \phi_m(x)\phi_m(x'),
    \end{align*}
    where $\{\lambda_m\}_{m=1}^{\infty}$ and $\{\phi_m\}_{m=1}^{\infty}$ are referred to as the eigenvalues and the eigenfeatures (or eigenfunctions) of $k(\cdot, \cdot)$ respectively.
\end{theorem}

\section{Posterior-prior RKHS Containment}\label{app:H0inHt}

\begin{lemma}\label{lem:H0inHt}
Let $k:\mathcal X\times \mathcal X\rightarrow \mathbb{R}$ be the reproducing kernel (prior covariance). Given the observation history $\mathcal{H}_t = \{A_t, y_{1:t}\}$ as in Section \ref{sec:Prob}, the $\alpha$-posterior kernel is, $k_t(x,x') = k(x,x') - k(x, A_t)^\top \big(k(A_t, A_t) + \lambda I_t\big)^{-1} k(A_t, x')$.
Then, the posterior RKHS $(\mathcal{H}_t)$ is contained inside the prior RKHS $(\mathcal{H}_0)$, i.e., $\mathcal{H}_t\subseteq \mathcal{H}_0$.
\end{lemma}
\begin{proof}
Using the fact that, $k(\cdot, \cdot)-k_t(\cdot, \cdot)$ is positive semi definite, a direct consequence of~\cite{berlinet2003reproducing}[Theorem 12] yields the result.
\end{proof}

\section{RKHS Norm Identity}\label{app:identity}

\begin{lemma}\label{lem:identity}
For any $f\in \mathcal{H}_t \subseteq \mathcal{H}_0$, it can be shown that:
$$
\|f\|_{k_{t}}^2= \|f\|_{k}^2 + \lambda^{-1}\sum_{s=1}^{t}(f(x_s))^2.
$$
\end{lemma}

\begin{proof}
Recall  that, $k_t(x,x')= k(x,x')- k(x,A_t)^\top (k(A_t,A_t)+\sigma^{-1} I_d)^{-1} k(A_t,x')$ and $\sigma^{-1}= \lambda$.  Since the kernel $k(s,t)$ is positive definite, continuous and square integrable, using Mercer's theorem (as given in Theorem \ref{thm:Mercer} in Appendix \ref{app:RKHS-Mercer}), there exists an orthonormal sequence of continuous eigenfunctions $\{\phi_j\}_{j=1}^{\infty}$ in $L^2(\mathcal X)$ with eigenvalues $\mu_1\geq \mu_2\ldots \geq 0$ and:
\begin{align}
    \int_{\mathcal X} k(s,t)\phi_j(s)ds= \mu_j \phi_j(t), \quad j={1,2,\ldots} \quad \text{and} \quad k(s,t)= \sum_{j=1}^{\infty} \mu_j \phi_j(s)\phi_j(t).
    \label{eq:rk0}
\end{align} 
The RKHS $\mathcal H_0$ spanned by kernel $k(\cdot,\cdot)$ consists of functions $g\in L^2(\mathcal X)$ satisfying $\sum_{j=1}^{\infty}{g_j^2}/{\mu_j} < \infty$, where $g_j= \langle g, \phi_j\rangle_{L^2(\mathcal X)}$. Moreover, for any $g\in\mathcal H_0$,  $\|g\|_{\mathcal H_0}^2 = \sum_{j=1}^{\infty}{g_j^2}/{\mu_j}$.

Now let us assume that $\{\gamma_i,\psi_i\}_{i=1}^{\infty}$ are the corresponding eigen system of the reproducing kernel $k_t(\cdot,\cdot)$. It follows that:
\begin{align}
    \int_{\mathcal X} k_t(x,x')\psi_j(x)dx= \gamma_j \psi_j(x'), \quad j={1,2,\ldots}
    \label{eq:rk1}
\end{align}

Considering the left hand side above and substituting the expression of $k_t$, it follows that:
\begin{align}
    \int_{\mathcal X} k_t(s,t)\psi_j(s)ds = \int_{\mathcal X} \left[k(x,x')- k(x,A_t)^\top (k(A_t,A_t)+\sigma^{-1} I_d)^{-1} k(A_t,x')\right]\psi_j(s)ds.
    \label{eq:rk2}
\end{align}

Since, $\psi_j \in L^2(\mathcal X)$ and $\{\phi_i\}_{i=1}^{\infty}$ is an orthonormal basis of $L^2(\mathcal X)$, we can express $\psi_j(\cdot) = \sum_{i=1}^{\infty} p^{(j)}_i  \phi_i(\cdot) $. Substituting this expression into~\eqref{eq:rk2} and interchanging integral and summation (assuming sufficient regularity of RKHS), we obtain from using the first expression in~\eqref{eq:rk0} that:
\begin{align}\label{eq:rk3}
\begin{split}
    &\int_{\mathcal X} k_t(x,x')\psi_j(x)dx\\
    &\qquad = \int_{\mathcal X} \left[k(x,x')- k(x,A_t)^\top (k(A_t,A_t)+\sigma^{-1} I_d)^{-1} k(A_t,x')\right]\sum_{i=1}^{\infty} p^{(j)}_i  \phi_i(x)dx
    \\
    &\qquad = \sum_{i=1}^{\infty} p^{(j)}_i \mu_i \phi_i(x') - k(x', A_t) (k(A_t,A_t)+\sigma^{-1} I_d)^{-1} \sum_{i=1}^{\infty} p^{(j)}_i \mu_i \phi_i(A_t)
    \\
    &\qquad = \sum_{i=1}^{\infty} p^{(j)}_i \mu_i \phi_i(x') - k(x', A_t) (k(A_t,A_t)+\sigma^{-1} I_d)^{-1} \sum_{i=1}^{\infty} p^{(j)}_i \mu_i \phi_i(A_t).
\end{split}
\end{align}

Now substituting~\eqref{eq:rk3} into~\eqref{eq:rk1} and taking inner product with $\phi_k(x')$, we have for all $k=1,2, \ldots$:
\begin{align}\label{eq:rk4}
\begin{split}
    &p^{(j)}_k \mu_k  - \mu_k \phi_k(A_t)^\top (k(A_t,A_t)+\sigma^{-1} I_d)^{-1} \sum_{i=1}^{\infty} p^{(j)}_i \mu_i \phi_i(A_t)\\
    &\qquad \qquad \qquad = \gamma_j \langle\psi_j(x'),\phi_k(x')\rangle = \gamma_j p^{(j)}_k.
\end{split}
\end{align}

We use the following matrix notation:
\begin{align}
\nonumber
    \mathbf{M}=\text{diag}(\mu_1,\mu_2,\ldots), \quad
        \mathbf{p}^{(j)}= (p^{j}_1,p^{j}_2,\ldots)^\top, \quad
    \Phi(A_t)= (\phi_1(A_t),\phi_2(A_t),\ldots)^\top.
\end{align}

Using the second expression in~\eqref{eq:rk0}, observe that $k(A_t,A_t)=  \Phi(A_t)^\top \mathbf{M} \Phi(A_t)$. Combined with this and using the above matrix notation,~\eqref{eq:rk4} can be rewritten as:
\begin{align}
    \left(\mathbf{M}  - \mathbf{M} \Phi(A_t)(\Phi(A_t)^\top \mathbf{M} \Phi(A_t)+\sigma^{-1} I_d)^{-1}\Phi(A_t)^\top \mathbf{M} \right) \mathbf{p}^{(j)} = \gamma_j \mathbf{p}^{(j)}.
\end{align}

With an application of Woodbury matrix identity, we have: 
$$\left(\mathbf{M}  - \mathbf{M} \Phi(A_t)(\Phi(A_t)^\top \mathbf{M} \Phi(A_t)+\sigma^{-1} I_d)^{-1}\Phi(A_t)^\top \mathbf{M} \right) = (\mathbf{M}^{-1} + \sigma \Phi(A_t)\Phi(A_t)^\top)^{-1},$$ 
and therefore:
\begin{align}
    \gamma_j \mathbf{p}^{(j)} = (\mathbf{M}^{-1} + \sigma \Phi(A_t)\Phi(A_t)^\top)^{-1} \mathbf{p}^{(j)}, \quad j=1,2,\ldots.
    \label{eq:rk5}
\end{align}

For any $f\in \mathcal H_t \subseteq\mathcal H_0$, observe that:
\begin{align}\label{eq:rk6}
\begin{split}
    \|f\|_{\mathcal H_t}^2 &=\sum_{j=1}^{\infty}  \frac{\langle f,\psi_j\rangle_{L^2(\mathcal X)}^2}{\gamma_j} = \sum_{j=1}^{\infty}  \frac{\langle f,\sum_{i=1}^{\infty} p^{(j)}_i  \phi_i \rangle_{L^2(\mathcal X)}^2}{\gamma_j}\\
    &= \sum_{j=1}^{\infty}  \frac{(\sum_{i=1}^{\infty} p^{(j)}_i f_i )^2}{\gamma_j}  = \sum_{j=1}^{\infty}  \frac{((\mathbf{p}^{(j)})^\top \mathbf{f})^2}{\gamma_j}\\
    & =   \mathbf{f}^\top \mathbf{P}\Gamma^{-1}\mathbf{P}^\top\mathbf{f},
\end{split}
\end{align}
where $\mathbf{f}=(f_1,f_2,\ldots)^\top$, $\mathbf{P}= (\mathbf{p}^{(1)},\mathbf{p}^{(2)},\ldots)^\top$, and $\Gamma= \text{diag}(\gamma_1,\gamma_2,\ldots)$. Using these notations, \eqref{eq:rk5} can be rewritten as, $(\mathbf{M}^{-1} + \sigma \Phi(A_t)\Phi(A_t)^\top)^{-1} \mathbf{P}=  \mathbf{P}\Gamma$. Also note using \eqref{eq:rk5} that, matrix $\mathbf{P}$ is an orthogonal matrix, which implies that $\mathbf{P}\Gamma^{-1}\mathbf{P}^\top = (\mathbf{M}^{-1} + \sigma \Phi(A_t)\Phi(A_t)^\top) $. Substituting the expression for $\mathbf{P}\Gamma^{-1}\mathbf{P}^\top$ in~\eqref{eq:rk6}, it follows that:
\begin{align}
    \|f\|_{\mathcal H_t}^2 = \mathbf{f}^\top (\mathbf{M}^{-1} + \sigma \Phi(A_t)\Phi(A_t)^\top) \mathbf{f} = \|f\|_{\mathcal H_0}^2 + \sigma \mathbf{f}^\top \Phi(A_t)\Phi(A_t)^\top \mathbf{f}.
\end{align}
Hence the assertion of Lemma \ref{lem:identity} follows, since $\sigma=  \lambda^{-1}$ and $\Phi(A_t)^\top \mathbf{f} = \sum_{j=1}^{\infty} f_j \phi_j(A_t) = f(A_t)$ because $\{\phi_j\}_{j=1}^{\infty}$ are orthonormal basis of $L^2(\mathcal X)$, which implies that $f(A_t)^\top f(A_t)= \sum_{s=1}^{t}(f(x_t))^2$.
\end{proof}

\section{Posterior Tail Probability Bound}~\label{app:LB}

\begin{lemma}~\label{lem:LB}
    With $\mathbb{P}_0$-probability of at least $1-\eta_2$ for any $\eta_2\in(0,1)$, we have:
    \begin{align}
        \mathbb{P} (\pmb \theta_t(x_0) > \pmb \theta_0(x_0)\mid \mathcal{F}_{t-1} ) \geq \frac{e^{-\iota_t^2}}{4\sqrt{\pi} \iota_t},
    \end{align}
    for $\iota_t^2 = \frac{(D+2)\alpha C(1-\alpha,\pmb\theta_0,k)}{ (1-\alpha)\eta_2}   t\epsilon_t^2$.
\end{lemma}
\begin{proof}    

\begin{align}
\begin{split}
\mathbb{P} (\pmb \theta_t(x_0) > \pmb \theta_0(x_0)  \mid \mathcal{F}_{t-1} ) &= \mathbb{P} \left(\frac{\pmb \theta_t(x_0) - \mu_{t-1}(x_0)}{\sigma_{t-1}(x_0)} > \frac{\pmb \theta_0(x_0)- \mu_{t-1}(x_0)}{\sigma_{t-1}(x_0)}\;\Big|\;\mathcal{F}_{t-1} \right)
    \\
    &\geq \frac{e^{-\beta_t^2}}{4\sqrt{\pi} \beta_t},
\end{split}
\end{align}
where $\beta_t=\frac{|\pmb \theta_0(x_0)- \mu_{t-1}(x_0)|}{\sigma_{t-1}(x_0)}$. 

Since, $\mu_{t-1}(\cdot) \in \mathcal{H}_{t-1}$, $\pmb \theta_0(\cdot)\in \mathcal{H}_{0}$, and $\mathcal{H}_{t-1}\subseteq \mathcal{H}_0$ (using Lemma \ref{lem:H0inHt}), the evaluation property of kernel $k_{t-1}^{\alpha}$ (see Appendix \ref{app:RKHS-Mercer}) implies:
\begin{align}\label{eq:GD1}
\begin{split}
    | \mu_{t-1}(x_0) -   \pmb \theta_0(x_0)| &= \langle ( \mu_{t-1}-\pmb \theta_0),k_{t-1}^{\alpha}(x_0,\cdot) \rangle_{k_{t-1}^{\alpha}} \\
      &\leq  \|\mu_{t-1}-\pmb \theta_0\|_{k_{t-1}^{\alpha}}  \|k_{t-1}^{\alpha}(x_0,\cdot)\|_{k_{t-1}^{\alpha}}
     \\
     &=\|\mu_{t-1}-\pmb \theta_0\|_{k_{t-1}^{\alpha}} \sqrt{k_{t-1}^{\alpha}(x_0,x
     _0)} 
     \\
     &=\|\mu_{t-1}-\pmb \theta_0\|_{k_{t-1}^{\alpha}} \sigma_{t-1}(x_0),
\end{split}
\end{align}
where the last inequality above is due to Cauchy-Schwarz inequality. It follows from Lemma~\ref{lem:identity} and the second result given by \eqref{eq:post-2} in Lemma~\ref{lem:post} that:
\begin{align}
\nonumber
\begin{split}
    \mathbb{E}\left[\frac{ |\mu_{t-1}(x_0) -   \pmb \theta_0(x_0)|^2}{\sigma_{t-1}^2(x_0)}\right] &\leq \mathbb{E} \left[\|\mu_{t-1}-\pmb \theta_0 \|_{k_{t-1}^{\alpha}}^2 \right]\\ 
    &=  \mathbb{E} \left[\|\mu_{t-1}-\pmb \theta_0\|_{k}^2 +\alpha\lambda^{-1} \sum_{s=1}^{t-1}(\mu_{t-1}(x_s)-\pmb \theta_0(x_s))^2  \right]
    \\
    &\leq 2 \mathbb{E} \left[\frac{1}{2} \|\pmb \theta_t -\pmb \theta_0\|_{k^\alpha}^2 +  \lambda^{-1} D_{\alpha}^{t-1}(\pmb \theta_t,\pmb \theta_0)  \right]
    \\
    &\leq 2  \left[\frac{(D+2)\alpha}{(1-\alpha)2} t\epsilon_t^2 + \frac{4C(1-\alpha,\pmb\theta_0,k) e^{ - \frac{\alpha}{4}  t\epsilon_t^2 } }{(1-\alpha)} \right]
    \\
    &\leq  \frac{(D+2)\alpha C(1-\alpha,\pmb\theta_0,k)}{ (1-\alpha)}    t\epsilon_t^2
,
\end{split}
\end{align}
where the first inequality above uses Jensen's inequality and the definition of the R\'enyi divergence between two multivariate Gaussian distributions. Now the assertion of Lemma \ref{lem:LB} follows using Markov's inequality.
\end{proof}

\section{Proof of Theorem \ref{thm:GRB}}\label{app:thm-GRB}
\begin{proof}
    Let $\overline{x}_t:= \arg\min_{x \notin \mathcal{C}_t} \sigma_{t-1}^2(x) $ and observe that:
    \begin{align}\label{eq:GB0}
        &\pmb \theta_0(x_0) -   \pmb \theta_0(x_t) = \left( \pmb \theta_0(x_0) -    \pmb \theta_0(\overline{x}_t) \right) + \left( \pmb \theta_0(\overline{x}_t) -   \pmb \theta_t(x_t) \right) + \left( \pmb \theta_t(x_t) -   \pmb \theta_0(x_t)\right) .
    \end{align}
    Using the definition of the unsaturated actions in Definition~\ref{def:US}, observe that the first term in~\eqref{eq:GB0} is bounded as:
    \begin{align}\label{eq:GB1}
    \begin{split}
        \sum_{t=1}^{T} \mathbb{E}\left[ \pmb \theta_0(x_0) -   \pmb \theta_0(\overline{x}_t) \right] &\leq  \sum_{t=1}^{T} {\mathcal{C}_t}\mathbb{E}[\sigma_{t-1}(\overline{x}_t)] \leq \sum_{t=1}^{T} \mathcal{C}_t \mathbb{E}[\sigma_{t-1}^2(\overline{x}_t)]^{\frac{1}{2}} \\
        &\leq \left(\sum_{t=1}^{T} \mathcal{C}_t^2 \right)^{\frac{1}{2}} \left(\sum_{t=1}^{T} \mathbb{E}[\sigma_{t-1}^2(\overline{x}_t)] \right)^{\frac{1}{2}},
    \end{split}
    \end{align} 
    where the last inequality follows from the Cauchy-Schwartz inequality. For the third summand in~\eqref{eq:GB0}, we add and subtract $\mu_{t-1}(x_t)$ to obtain
    \(
      \pmb \theta_t(x_t) -   \pmb \theta_0(x_t) = 
      \pmb \theta_t(x_t) - \mu_{t-1}(x_t)+ \mu_{t-1}(x_t)-  \pmb \theta_0(x_t).\)
    Since, $\mu_{t-1}(\cdot) \in \mathcal{H}_{t-1}$, $\pmb \theta_0(\cdot)\in \mathcal{H}_{0}$, and $\mathcal{H}_{t-1}\subseteq \mathcal{H}_0$ from Lemma \ref{lem:H0inHt}, the evaluation property of kernel $k_{t-1}^{\alpha}(\cdot, \cdot)$ (see Appendix \ref{app:RKHS-Mercer}) implies:
    \begin{align}\label{eq:GB1.5}
    \begin{split}
        \mu_{t-1}(x_t) -   \pmb \theta_0(x_t) &= \langle (\mu_{t-1}-\pmb \theta_0), k_{t-1}^{\alpha}(x_t,\cdot) \rangle_{k_{t-1}^{\alpha}}\\
        &\leq  \|\mu_{t-1}-\pmb \theta_0\|_{k_{t-1}^{\alpha}}  \| k_{t-1}^{\alpha}(x_t,\cdot)\|_{k_{t-1}^{\alpha}}\\
        &=\|\mu_{t-1}-\pmb \theta_0\|_{k_{t-1}^{\alpha}} \sqrt{k_{t-1}^{\alpha}(x_t,x
     _t)},
    \end{split}
\end{align}
where the second inequality above, once again follows from Cauchy-Schwarz inequality. Since \(\sum_{t=1}^{T} \mathbb{E}\left[\pmb \theta_t(x_t) - \mu_{t-1}(x_t)\right]= 0\) (by definition), it follows using the observation in~\eqref{eq:GB1.5} above that:
    \begin{align}\label{eq:GB2}
    \begin{split}
        \sum_{t=1}^T \mathbb{E} \left[\pmb \theta_t(x_t) -   \pmb \theta_0(x_t)\right] & \leq  \sum_{t=1}^T \mathbb{E} \left[  \left|\mu_{t-1}(x_t) -   \pmb \theta_0(x_t)\right|  \right]\\ 
       & \leq  \sum_{t=1}^T \mathbb{E} \left[ \sigma_{t-1}(x_t) \|\mu_{t-1} -\pmb \theta_0 \|_{k_{t-1}^{\alpha}} \right] 
       \\
       &\leq  \mathbb{E} \left[\left( \sum_{t=1}^T \sigma_{t-1}^2(x_t) \right)^{\frac{1}{2}} \left(  \sum_{t=1}^T \|\mu_{t-1} -\pmb \theta_0 \|_{k_{t-1}^{\alpha}}^2 \right)^{\frac{1}{2}}\right] 
       \\
       &\leq   \left( \mathbb{E} \left[ \sum_{t=1}^T \sigma_{t-1}^2(x_t)\right]  \right)^{\frac{1}{2}} \left(   \sum_{t=1}^T  \mathbb{E} \left[\| \mu_{t-1} -\pmb \theta_0 \|_{k_{t-1}^{\alpha}}^2 \right] \right)^{\frac{1}{2}},
    \end{split}
    \end{align}
where the third and fourth inequality above are due to Cauchy-Schwarz inequality. 
Combined with the fact that, $\pmb \theta_t(x_t) > \pmb \theta_t(\overline{x}_t)$ (by definition of $x_t$), the second term in~\eqref{eq:GB0} can be bounded in a similar way to obtain:
    \begin{align}
        \sum_{t=1}^{T} \mathbb{E}\left[ \pmb \theta_0(\overline{x}_t) -   \pmb \theta_t(x_t) \right] \leq   \left( \mathbb{E} \left[ \sum_{t=1}^T \sigma_{t-1}^2(\overline{x}_t)\right]  \right)^{\frac{1}{2}} \left(   \sum_{t=1}^T  \mathbb{E} \left[\|\mu_{t-1} -\pmb \theta_0 \|_{k_{t-1}^{\alpha}}^2 \right] \right)^{\frac{1}{2}}.
        \label{eq:GB3}
    \end{align}
It follows from Lemma~\ref{lem:identity} and the second result in Lemma~\ref{lem:post} that:
\begin{align}\label{eq:GB4}
\begin{split}
    \sum_{t=1}^T  \mathbb{E} \left[\|\mu_{t-1}-\pmb \theta_0 \|_{k_{t-1}^{\alpha}}^2 \right] 
    &= \sum_{t=1}^T  \mathbb{E} \left[\|\mu_{t-1}-\pmb \theta_0\|_{k^\alpha}^2 +\alpha\lambda^{-1} \sum_{s=1}^{t-1}(\mu_{t-1}(x_s)-\pmb \theta_0(x_s))^2  \right]\\
    &\leq 2\sum_{t=1}^T  \mathbb{E} \left[\frac{1}{2} \|\pmb \theta_t -\pmb \theta_0\|_{k^\alpha}^2 +  \lambda^{-1} D_{\alpha}^{t-1}(\pmb \theta_t,\pmb \theta_0)  \right]\\
    &\leq 2 \sum_{t=1}^T  \left[\frac{(D+2)\alpha}{2(1-\alpha)} t\epsilon_t^2 + \frac{4C(1-\alpha,\pmb\theta_0,k) e^{ - \frac{\alpha}{4}  t\epsilon_t^2 } }{(1-\alpha)} \right]\\
    &\leq  \frac{(D+2)\alpha C(1-\alpha,\pmb\theta_0,k)}{ (1-\alpha)}  \sum_{t=1}^T  t\epsilon_t^2,
\end{split}
\end{align}
where the first inequality above follows from Jensen's inequality and the definition of R\'enyi divergence between two multivariate Gaussian distributions. 

The only term that remains to be analyzed now is $\mathbb{E} \left[ \sigma_{t-1}^2(\overline{x}_t)\right]$. Observe that:
\begin{align}\label{eq:GB5}
\begin{split}
    \mathbb{E} \left[ \sigma_{t-1}^2(x_t)\mid \mathcal{F}_{t-1}\right] &\geq  \mathbb{E} \left[ \sigma_{t-1}^2(x_t)\mid \mathcal{F}_{t-1}, x_t \notin \mathcal{C}_t\right] P(x_t \notin \mathcal{C}_t\mid \mathcal{F}_{t-1} )\\
    &\geq \sigma_{t-1}^2(\overline{x}_t) P(x_t \notin \mathcal{C}_t \mid \mathcal{F}_{t-1} ),
\end{split}
\end{align}
where the second inequality above uses the definition of $\overline{x}_t$ and the fact that it is completely determined by $\mathcal{F}_{t-1}$.
Note that, if we can establish a lower bound on  $P(x_t \notin \mathcal{C}_t \mid \mathcal{F}_{t-1} )$, then we can upper bound $\mathbb{E} \left[ \sigma_{t-1}^2(\overline{x}_t)\right]$ by $\mathbb{E} \left[ \sigma_{t-1}^2(x_t)\right]$ up to some time dependent term. 
Since $x_0$ is always unsaturated, therefore if, $\pmb \theta_t(x_0) > \pmb \theta_t(x)$ for all saturated actions, i.e., $ \forall x\in \mathcal{C}_t$, then one of the unsaturated actions must be played. Consequently:
\begin{align}\label{eq:eqLB2}
\begin{split}
    &\mathbb{P}(x_t \notin \mathcal{C}_t\mid \mathcal{F}_{t-1} ) \geq \mathbb{P}( \pmb \theta_t(x_0) > \pmb \theta_t(x), \forall x\in \mathcal{C}_t \mid \mathcal{F}_{t-1} )\\
    &\qquad \geq  \mathbb{P}( \{ \forall x\in \mathcal{C}_t: \pmb \theta_t(x_0) > \pmb \theta_t(x) \} , \{\forall x\in \mathcal X: \pmb \theta_0(x) \geq \pmb \theta_t(x) -  2 \mathcal{C}_t \sigma_{t-1}(x)\} \mid \mathcal{F}_{t-1} ) 
    \\
    &\qquad \geq  \mathbb{P}( \{  \pmb \theta_t(x_0) > \pmb \theta_0(x_0)   \}  , \{\forall x\in \mathcal X: \pmb \theta_0(x) \geq \pmb \theta_t(x) - 2 \mathcal{C}_t \sigma_{t-1}(x)\} \mid \mathcal{F}_{t-1} )
    \\
    &\qquad \geq  \mathbb{P}\left( \{  \pmb \theta_t(x_0) > \pmb \theta_0(x_0)   \} , \{\sup_{x\in\mathcal X} \frac{\pmb\theta_t(x)-\pmb \theta_0(x)}{\sigma_{t-1}(x)}   \leq 2 \mathcal{C}_t \} \mid \mathcal{F}_{t-1} \right)\\
    &\qquad \geq \mathbb{P} \left(\pmb \theta_t(x_0) > \pmb \theta_0(x_0)  | \mathcal{F}_{t-1} \right)  -  \mathbb{P}\left(\sup_{x\in\mathcal X} \frac{\pmb\theta_t(x)-\pmb \theta_0(x)}{\sigma_{t-1}(x)}   \geq 2 \mathcal{C}_t  \mid \mathcal{F}_{t-1} \right),    
\end{split}
\end{align}
where the third inequality above uses the definition of saturated arms and the last inequality above uses union bound. 

Now using standard supremum inequality, i.e., $\sup(f+g) \le \sup f +\sup g $ and the result in~\eqref{eq:GB1.5} observe that,
\( \sup_{x\in\mathcal X} \frac{\pmb\theta_t(x)-\pmb \theta_0(x)}{\sigma_{t-1}(x)} \le  \sup_{x\in\mathcal X} \frac{\pmb\theta_t(x)-\pmb \mu_{t-1}(x)}{\sigma_{t-1}(x)} + \|\mu_{t-1}-\pmb \theta_0\|_{k_{t-1}^{\alpha}}. \) Also, after application of Jensen's and then Markov's inequalities, it follows from \eqref{eq:post-2} in Lemma~\ref{lem:post} 
that, $\mathbb{P}_0(\|\mu_{t-1}-\pmb \theta_0\|_{k_{t-1}^{\alpha}} \geq \eta_1 \mathcal{C}_t ) \leq \frac{1}{\eta_1^2}$ for any $\eta_1\geq 1$.  This implies:
\begin{align}
\mathbb{P}\left(\sup_{x\in\mathcal X} \frac{\pmb\theta_t(x)-\pmb \theta_0(x)}{\sigma_{t-1}(x)}   \geq 2 \mathcal{C}_t\;\Big|\;\mathcal{F}_{t-1} \right) \leq \mathbb{P}\left( \sup_{x\in\mathcal X} \frac{\pmb\theta_t(x)-\mu_{t-1}(x)}{\sigma_{t-1}(x)}   \geq (2-\eta_1) \mathcal{C}_t\;\Big|\;\mathcal{F}_{t-1} \right),
\label{eq:UB}
\end{align}
with $\mathbb{P}_0$-probability of at least $1-\eta_1^{-2}$.
Since, $\frac{\pmb\theta_t(x)-\mu_{t-1}(x)}{\sigma_{t-1}(x)}$ for any $x$ is a standard Gaussian random variable given $\mathcal{F}_{t-1}$, it follows using Gaussian tail inequality on the right hand side of~\eqref{eq:UB} that:
\begin{align}
\mathbb{P}\left(\sup_{x\in\mathcal X} \frac{\pmb\theta_t(x)-\pmb \theta_0(x)}{\sigma_{t-1}(x)}   \geq 2 \mathcal{C}_t\;\Big|\;\mathcal{F}_{t-1} \right) \leq e^{-\frac{(2-\eta_1)^2 \mathcal{C}_t^2}{2}},
\end{align}
with $\mathbb{P}_0$-probability of at least $1-\eta_1^{-2}$.
Combined with the observation in Lemma~\ref{lem:LB}, it follows that:
\begin{align}
    \mathbb{P}(  x_t \notin \mathcal{C}_t \mid \mathcal{F}_{t-1}  ) \geq \frac{e^{-\iota_t^2}}{4\sqrt{\pi} \iota_t}
     - e^{ - ((2-\eta_1)^2)\frac{D\alpha}{\lambda(1-\alpha)}  t\epsilon_t^2 }:=\overline{p}(t),
     \label{eq:UB1}
\end{align}
with $\mathbb{P}_0$-probability of at least $1-\eta_1^{-2}-\eta_2^{-1}$
    for $\iota_t^2 = \frac{(D+2)\alpha C(1-\alpha,\pmb\theta_0,k)\eta_2 }{ (1-\alpha)}   t\epsilon_t^2$.
It would suffice to assume $\alpha t\epsilon_t^2 <1$ to lower bound $\bar p(t)$ because the function $e^{-x^2}(x^{-1}-1)>0$ when $x<1$.
Now choosing a sufficiently large $t=t_1$ and $D$, we can ensure that the lower bound $\overline{p}(t_1) = \overline{p}_0$ is positive. 
Now denoting the above high probability event $\{\mathbb{P}(  x_t \notin \mathcal{C}_t \mid \mathcal{F}_{t-1}  ) \geq \overline{p}(t_1)\}$ as $\mathbf{E}$, observe that:
\[\mathbb{E}\left[\mathbb{E} \left[ \sigma_{t-1}^2(x_t)\mid\mathcal{F}_{t-1}\right]  \right] \geq \mathbb{E}\left[\mathbb{E} \left[ \sigma_{t-1}^2(x_t)\mid\mathcal{F}_{t-1}\right] \mid\mathbf{E} \right] ]\mathbb{P}_0(\mathbf{E}) \geq \mathbb{E}\left[\sigma_{t-1}^2(\overline{x}_t)\right] \overline{p}(t_1) \mathbb{P}_0(\mathbf{E}). \]
Therefore, we have:
\begin{align}
\mathbb{E} \left[\sum_{t=1}^{T} \sigma_{t-1}^2(\overline{x}_t)\right] \leq [\overline{p}_0(1-\eta_1^{-2}-\eta_2^{-1})]^{-1}\mathbb{E} \left[ \sum_{t=1}^{T} \sigma_{t-1}^2(x_t)\right]. 
\label{eq:LB}
\end{align}

Now, putting the pieces together from \eqref{eq:GB0}, \eqref{eq:GB1}, \eqref{eq:GB2}, \eqref{eq:GB3}, \eqref{eq:GB4}, \eqref{eq:GB5}, and \eqref{eq:LB}, we have:
\begin{align}
\begin{split}
    &\mathbb{E}[\mathrm{Regret}(T)]\\
    &\quad\leq  \left(\sum_{t=1}^{T} \mathcal{C}_t^2 \right)^{\frac{1}{2}} \left(\sum_{t=1}^{T} \mathbb{E}[\sigma_{t-1}^2(\overline{x}_t)] \right)^{\frac{1}{2}}\\
    &\quad + \left( \mathbb{E} \left[ \sum_{t=1}^T \sigma_{t-1}^2(\overline{x}_t)\right]  \right)^{\frac{1}{2}} \left(  \frac{(D+2)\alpha C(1-\alpha,\pmb\theta_0,k)}{ (1-\alpha)}  \sum_{t=1}^T  t\epsilon_t^2\right)^{\frac{1}{2}}\\
    &\quad + \left( \mathbb{E} \left[ \sum_{t=1}^T \sigma_{t-1}^2( x_t)\right]  \right)^{\frac{1}{2}} \left(  \frac{(D+2)\alpha C(1-\alpha,\pmb\theta_0,k)}{ (1-\alpha)}  \sum_{t=1}^T  t\epsilon_t^2\right)^{\frac{1}{2}}\\
    &\quad \leq \left(\mathbb{E}\left[\sum_{t=1}^{T}\sigma_{t-1}^{2}(x_t)\right]\right)^{\frac{1}{2}}\bigg\{[\overline{p}_0(1-\eta_1^{-2}-\eta_2^{-1})]^{-1}\bigg[\left(\sum_{t=1}^{T}\mathcal{C}_t^2\right)^{\frac{1}{2}}\\
    &\quad + \left(\frac{(D+2)\alpha C(1-\alpha, \lambda, \pmb{\theta}_0)}{(1-\alpha)}\sum_{t=1}^{T}t\epsilon_t^2\right)^{\frac{1}{2}} \bigg] + \left(\frac{(D+2)\alpha C(1-\alpha, \lambda, \pmb{\theta}_0)}{(1-\alpha)}\sum_{t=1}^{T}t\epsilon_t^2\right)^{\frac{1}{2}}\bigg\}.\\
\end{split} 
\end{align}
Substituting $\mathcal{C}_t^2=\frac{D\alpha}{(1-\alpha)} t\epsilon_t^2$, we have:
\begin{align}\label{eq:regret-bound-final}
\begin{split}
    \mathbb{E}[\mathrm{Regret}(T)] &\leq \left\{[\overline{p}_0(1-\eta_1^{-2}-\eta_2^{-1})]^{-\frac{1}{2}}\left[\left(\frac{D}{(1-\alpha)}\right)^{\frac{1}{2}} + \left(\frac{(D+2) C(1-\alpha, \lambda, \pmb{\theta}_0)}{(1-\alpha)}\right)^{\frac{1}{2}}\right]\right.\\
    &\left.+ \left(\frac{(D+2) C(1-\alpha, \lambda, \pmb{\theta}_0)}{(1-\alpha)}\right)^{\frac{1}{2}}\right\} \left( \mathbb{E} \left[ \sum_{t=1}^T \alpha \sigma_{t-1}^2(x_t)\right]  \right)^{\frac{1}{2}} \left(\sum_{t=1}^T  t\epsilon_t^2 \right)^{\frac{1}{2}}.\\
\end{split}
\end{align}
Now using the definition of $\sigma^2_{t-1}(\cdot)$, the result in \eqref{eq:regret-bound-final-alpha-theorem} follows.
\end{proof}

\section{Proof of Lemma \ref{lem:post}
}\label{app:Post}
\begin{proof}
    Let us first define a set for any $D>0$, $\lambda>0$, $j>0$, $\alpha\in(0,1)$, and given $A_t$ as,
 \(   F_{t,\epsilon_t} : = \left\{\pmb \theta \in \Theta: \frac{1}{2} \|\pmb \theta -\pmb \theta_0\|_k^2 + D_{\alpha}^{(t)}(\pmb \theta,\pmb \theta_0) \geq \frac{(D+j)\alpha}{(1-\alpha)2} t\epsilon_t^2  \right\}.\)
Note that, $F_{t,\epsilon_t}$ is adapted to $\mathcal{F}_t$ because the definition of $D_{\alpha}^{(t)}(\pmb \theta,\pmb \theta_0)$ requires knowledge of $A_t$. The $\alpha$-posterior measure of the set $F_{t,\epsilon_t}$ can be defined as:
\begin{align}
\Pi_{t,\alpha}(F_{t,\epsilon_t} \mid \mathcal{F}_t)&= \frac{ \int_{F_{t,\epsilon_t}}  \prod_{s=1}^{t} [p_{\pmb \theta}(y_s \mid x_s)]^{\alpha}\Pi(d\pmb \theta) } { \int_\Theta  \prod_{s=1}^{t} [p_{\pmb \theta}(y_s \mid x_s)]^{\alpha} \Pi(d\pmb \theta)  }
    = \frac{\int_{F_{t,\epsilon_t}} \Pi(d\pmb \theta) e^{-\alpha \ell_{t}(\pmb \theta,\pmb \theta_0)} } {\int_{\Theta} \Pi(d\pmb \theta) e^{-\alpha \ell_{t}(\pmb \theta,\pmb \theta_0)} },
    \label{eq:MABDE3}
\end{align}
where $\ell_{t}(\pmb \theta,\pmb \theta_0)= \log \frac{\prod_{s=1}^{t} p_0(y_s \mid x_s)}{\prod_{s=1}^{t} p_{\pmb \theta}(y_s \mid x_s) }$.
Now define a set:
$$
S_{t,\epsilon_t} = \left\{ \mathcal{D}_t : \int_{\Theta} \tilde \Pi(d\pmb \theta) e^{-\alpha \ell_{t}(\pmb \theta,\pmb \theta_0)} \leq   e^{-\frac{(D+j)\alpha}{4} t\epsilon_t^2}  \right\},
$$
where for any $B\subseteq \pmb \theta$, $\tilde \Pi(B) = \frac{\Pi(B\cap \pmb \theta)}{\Pi(B\left(\pmb \theta_0,\epsilon_t \right))}$ and $B\left(\pmb \theta_0,\epsilon_t \right)$ is as defined in Assumption~\ref{ass:prior}. 
Now observe that:
\begin{align}
\mathbb{E} [\Pi_{t,\alpha}(F_{t,\epsilon_t} \mid \mathcal{F}_{t}) |A_t ] =  \mathbb{P}_0^{(t)} (S_{t,\epsilon_t}  |A_t) + \mathbb{E} \left[  \frac{\int_{F_{t,\epsilon_t}} \Pi(d\pmb \theta) e^{-\alpha \ell_{t}(\pmb \theta,\pmb \theta_0)} } {\int_{\Theta} \Pi(d\pmb \theta) e^{-\alpha \ell_{t}(\pmb \theta,\pmb \theta_0)} } \text{I}(S_{t,\epsilon_t}^c)   |A_t \right],
\label{eq:MABDE4}
\end{align}
First, let us analyze the second term in~\eqref{eq:MABDE4} above. Note that on the set $S_{t,\epsilon_t}^c$:
$$
\int_{\Theta} \Pi(d\pmb \theta) e^{-\alpha \ell_{t}(\pmb \theta,\pmb \theta_0)} \geq \Pi(B\left(\pmb \theta_0,\epsilon_t \right))  e^{-\frac{(D+j)\alpha}{4} t \epsilon_t^2}.
$$
Therefore, using  Assumption~\ref{ass:prior}, it follows that:
\begin{align}\label{eq:MABDE5}
\begin{split}
\mathbb{E} \left[  \frac{\int_{F_{t,\epsilon_t}} \Pi(d\pmb \theta) e^{-\alpha \ell_{t}(\pmb \theta,\pmb \theta_0)} } {\int_{\Theta} \Pi(d\pmb \theta) e^{-\alpha \ell_{t}(\pmb \theta,\pmb \theta_0)}  } \text{I}(S_{t,\epsilon_t}^c) \Big | A_t   \right] &\leq e^{\frac{(D+j)\alpha}{4} t \epsilon_t^2}
 \mathbb{E} \left[  \frac{{\int_{F_{t,\epsilon_t}} \Pi(d\pmb \theta) e^{-\alpha \ell_{t}(\pmb \theta,\pmb \theta_0)}}}{{\Pi(B\left(\pmb \theta_0,\epsilon_t \right)) }} \Big | A_t \right] 
 \\
 &\leq e^{(\frac{j\alpha}{4} + \frac{D\alpha}{2}) t \epsilon_t^2  }
 \mathbb{E} \left[  {\int_{F_{t,\epsilon_t}} \Pi(d\pmb \theta) e^{-\alpha \ell_{t}(\pmb \theta,\pmb \theta_0)}}  \Big | A_t  \right].
\end{split}
\end{align}
Now, it follows from the definition of $\alpha$-R\'enyi divergence that:
$$
\mathbb{E}\left[e^{-\alpha \ell_{t}(\pmb \theta,\pmb \theta_0)} \mid A_t\right]= 
e^{-(1-\alpha) D_{\alpha}^{(t)}(\pmb \theta,\pmb \theta_0)}.
$$
Hence, using Fubini's theorem and the observation above, it follows that:
\begin{align}\label{eq:MABDE6}
\nonumber
\mathbb{E}\left[\int_{F_{t,\epsilon_t}} \Pi(d\pmb \theta) e^{-\alpha \ell_{t}(\pmb \theta,\pmb \theta_0)} \;\Big|\;A_t\right] &=\int_{F_{t,\epsilon_t}} \Pi(d\pmb \theta) e^{-(1-\alpha)  D_{\alpha}^{(t)}(\pmb \theta,\pmb \theta_0)}
     \\
\nonumber
&\quad \leq \int_{F_{t,\epsilon_t}} \Pi(d\pmb \theta) e^{-(1-\alpha)  \left(\frac{(D+j)\alpha}{(1-\alpha)2} t\epsilon_t^2  - \frac{1}{2} \|\pmb \theta-\pmb \theta_0\|_{k^\alpha}^2 \right)}
\\
&\quad \leq 
     e^{- \frac{(D+j) \alpha}{2}  t\epsilon_t^2 } \int_{F_{t,\epsilon_t}} \Pi(d\pmb \theta) e^{\frac{(1-\alpha) }{2} \|\pmb \theta-\pmb \theta_0\|_{k^\alpha}^2}, 
\end{align}
where the penultimate inequality is due to the definition of  $F_{t,\epsilon_t}$.
Substituting~\eqref{eq:MABDE6} into~\eqref{eq:MABDE5} yields:
\begin{align}
\mathbb{E} \left[  \frac{\int_{F_{t,\epsilon_t}} \Pi(d\pmb \theta) e^{-\alpha \ell_{t}(\pmb \theta,\pmb \theta_0)}} {\int_{\Theta} \Pi(d\pmb \theta) e^{-\alpha \ell_{t}(\pmb \theta,\pmb \theta_0)} } \text{I}(S_{t,\epsilon_t}^c)   \Big | A_t \right] 
&\leq e^{ -\frac{j \alpha }{4} t\epsilon_t^2 } \int_{F_{t, \epsilon_t}} \Pi(d\pmb \theta) e^{\frac{(1-\alpha) }{2} \|\pmb \theta-\pmb \theta_0\|_{k^\alpha}^2}.
\label{eq:MABDE7}
\end{align}
Next, we analyze the first term in~\eqref{eq:MABDE4} above.  It follows from Markov's inequality that:
\begin{align}\label{eq:MABDE8}
\begin{split}
&\mathbb{P}_0^{(t)} \left( \left[ \int_{\Theta} \tilde\Pi(d\pmb \theta) e^{-\alpha \ell_{t}(\pmb \theta,\pmb \theta_0)}\right]^{-\frac{1}{\alpha}} \geq e^{ \frac{D+j}{4} t\epsilon_t^2 } \Big | A_t \right)\\
&\quad \leq  {e^{ - \frac{D+j}{4} t\epsilon_t^2 } }\mathbb{E} \left( \left[ \int_{\Theta} \tilde \Pi(d\pmb \theta) e^{-\alpha \ell_{t}(\pmb \theta,\pmb \theta_0)} \right]^{-\frac{1}{\alpha}} \Big | A_t  \right)
\\
&\quad \leq  {e^{ -  \frac{D+j}{4}  t\epsilon_t^2 } }  \int_{\Theta} \tilde \Pi(d\pmb \theta) \mathbb{E} \left( e^{\ell_{t}(\pmb \theta,\pmb \theta_0)}\;\Big|\;A_t \right) 
\\
&\quad =   {e^{ -  \frac{D+j}{4}  t\epsilon_t^2 } } \mathbb{E}\left[ \int_{\Theta} \tilde \Pi(d\pmb \theta) e^{ D_2^{(t)} \left( \pmb \theta_0, \pmb \theta \right)} \right]
\\
&\quad \leq   {e^{ - \frac{(D+j)\alpha }{4}  t\epsilon_t^2 } e^{ \frac{D \alpha }{4} t\epsilon_t^2 } }
=  e^{ - \frac{j\alpha}{4}  t\epsilon_t^2 },
\end{split}
\end{align}
where the second inequality is due to Jensen's inequality and Fubini's theorem. The penultimate inequality above uses the definition of the set $B\left(\pmb \theta_0,\epsilon_t \right)$ and the fact that $\alpha\in(0,1)$.
Combining~\eqref{eq:MABDE7} and~\eqref{eq:MABDE8}, it follows from~\eqref{eq:MABDE4}  that for all $j>0$ and $\alpha\in (0,1)$:
\begin{align}\label{eq:MABDE9}
\begin{split}
&\mathbb{E}\left[\Pi_{t,\alpha}\left( \frac{1}{2}\|\pmb \theta-\pmb \theta_0\|_{k^\alpha}^2 + D_{\alpha}^{(t)}(\pmb \theta,\pmb \theta_0) \geq \frac{(D+j)\alpha}{(1-\alpha)2} t\epsilon_t^2 \Big|\;\mathcal{F}_{t} \right)\Big | A_t \right]\\  
&\qquad = \mathbb{E} [\Pi_{t,\alpha}(F_{t,\epsilon_t} \mid \mathcal{F}_{t}) | A_t ]\\
&\qquad \leq 2 e^{ - \frac{j\alpha}{4}  t\epsilon_t^2 }  \int_{F_{t, \epsilon_t}} \Pi(d\pmb \theta) e^{\frac{(1-\alpha) }{2} \|\pmb \theta-\pmb \theta_0\|_{k^\alpha}^2}.
\end{split}
\end{align}
Now the first assertion in \eqref{eq:post-1} follows by using the observation that:
\begin{align}\label{eq:MABDE9a}
\begin{split}
    \int_{F_{t, \epsilon_t}} \Pi(d\pmb \theta) e^{\frac{(1-\alpha) }{2} \|\pmb \theta-\pmb \theta_0\|_{k^\alpha}^2} &\leq \int_{F_{t, \epsilon_t}} \Pi(d\pmb \theta) e^{\frac{(1-\alpha)\alpha }{2} \|\pmb \theta-\pmb \theta_0\|_{k}^2}\\
    &\leq \int_{F_{t, \epsilon_t}} \Pi(d\pmb \theta) e^{{ (1-\alpha)} (\|\pmb \theta\|_k^2+\|\pmb \theta_0\|_k^2)}\\
    &= e^{{ (1-\alpha)} \|\pmb \theta_0\|_k^2} \int \Pi(d\pmb \theta)e^{{ (1-\alpha)} \|\pmb \theta\|_k^2}:= C(1-\alpha,\pmb \theta_0,k).
\end{split}
\end{align} 
    
For the second assertion in \eqref{eq:post-2}, note that the right hand side above is non-increasing in $j$. Therefore, it follows from the inequality above that, for all $s\geq 1$:
\begin{align}\label{eq:MABDE10}
\begin{split}
&\mathbb{E}\left[\Pi_{t,\alpha}\left(\frac{1}{2}\|\pmb \theta-\pmb \theta_0\|_{k^\alpha}^2 + D_{\alpha}^{(t)}(\pmb \theta,\pmb \theta_0) \geq \frac{(D+s)\alpha}{(1-\alpha)2} t\epsilon_t^2 \;\Big|\;\mathcal{F}_{t}\right) \Big | A_t  \right]\\
&\qquad \leq  2 C(1-\alpha,\pmb \theta_0,k) e^{ - \frac{(s-1)\alpha}{4}  t\epsilon_t^2 } .
\end{split}
\end{align}
Since $\frac{1}{2}\|\pmb \theta-\pmb \theta_0\|_{k^\alpha}^2 + D_{\alpha}^{(t)}(\pmb \theta,\pmb \theta_0)>0$, it is straightforward to see that:
\begin{align*}
\begin{split}
&\mathbb{E}\left[\frac{1}{2}\|\pmb \theta-\pmb \theta_0\|_{k^\alpha}^2 + D_{\alpha}^{(t)}(\pmb \theta,\pmb \theta_0)\;\Big|\;\mathcal{F}_{t} \right] = \int_{0}^{\infty} \mathbb{E}\left[\text{I}\left\{\frac{1}{2}\|\pmb \theta-\pmb \theta_0\|_{k^\alpha}^2 + D_{\alpha}^{(t)}(\pmb \theta,\pmb \theta_0) \geq u \right\}\;\Big|\;\mathcal{F}_{t} \right] du
\\
&\qquad \qquad\qquad \leq \frac{(D+2)\alpha}{(1-\alpha)2} t\epsilon_t^2 + \int_{\frac{(D+2)\alpha}{(1-\alpha)2} t\epsilon_t^2}^{\infty} \mathbb{E}\left[\text{I}\left\{\frac{1}{2}\|\pmb \theta-\pmb \theta_0\|_{k^\alpha}^2 + D_{\alpha}^{(t)}(\pmb \theta,\pmb \theta_0) \geq u \right\}\;\Big|\;\mathcal{F}_{t} \right]  du.
\end{split}
\end{align*}
Now using Fubini's theorem, we have:
\begin{align}
\begin{split}
&\mathbb{E}\left[\mathbb{E}\left[\frac{1}{2}\|\pmb \theta-\pmb \theta_0\|_{k^\alpha}^2 + D_{\alpha}^{(t)}(\pmb \theta,\pmb \theta_0)\;\Big|\; \mathcal{F}_{t} \right]\right]
\\
& \leq \frac{(D+2)\alpha}{(1-\alpha)2} t\epsilon_t^2 + \int_{\frac{(D+2)\alpha}{(1-\alpha)2} t\epsilon_t^2}^{\infty} \mathbb{E}\left[\mathbb{E}\left[\text{I}\left\{ \frac{1}{2}\|\pmb \theta-\pmb \theta_0\|_{k^\alpha}^2 + D_{\alpha}^{(t)}(\pmb \theta,\pmb \theta_0) \geq u \right\} \;\Big|\; \mathcal{F}_{t} \right]\right]  du
\\
& = \frac{(D+2)\alpha}{(1-\alpha)2} t\epsilon_t^2 + \int_{2}^{\infty} \mathbb{E}\left[\mathbb{E}\left[\text{I}\left\{\frac{1}{2}\|\pmb \theta-\pmb \theta_0\|_{k^\alpha}^2 + D_{\alpha}^{(t)}(\pmb \theta,\pmb \theta_0) \geq \frac{(D+s)\alpha}{(1-\alpha)2} t\epsilon_t^2 \right\} \;\Big|\; \mathcal{F}_{t} \right]\right] \frac{\alpha t\epsilon_t^2}{(1-\alpha)2}  ds
\\
& \leq \frac{(D+2)\alpha}{(1-\alpha)2} t\epsilon_t^2 + C(1-\alpha,\pmb \theta_0,k) \frac{\alpha t\epsilon_t^2}{(1-\alpha)}   e^{ - \frac{\alpha}{4}  t \epsilon_t^2 }  \int_{2}^{\infty}  e^{ - \frac{(s-2)\alpha}{4}  t \epsilon_t^2 } ds 
\\
&= \frac{(D+2)\alpha}{(1-\alpha)2} t\epsilon_t^2 + C(1-\alpha,\pmb \theta_0,k) \frac{4e^{ - \frac{\alpha}{4}  t \epsilon_t^2 } }{(1-\alpha)} 
\end{split}
\end{align}
where the last inequality above is due to~\eqref{eq:MABDE10}. Now the result follows by using~\eqref{eq:MABDE9a}.
\end{proof}

\section{Cameron-Martin Lower Bound}\label{app:cameron-martin-formula}

\begin{lemma}
\label{lemma:cameron-martin-formula}
Let $W^{a} \sim \mathrm{GP}(0, k_a(\cdot, \cdot))$ under the conditional law $\pmb{\theta}\mid a$ denoted by $\Pi_a(\cdot)$. If $\pmb{\theta}_0(\cdot) \in \mathcal{H}_a$, where $\mathcal{H}_a$ is the associated RKHS of $W^{a}$, then for every Borel set $A\subset \mathcal{C}[0, 1]^d$:
\begin{align*}
    \Pi_a\left(W^{a} \in A - \pmb{\theta}_0\right) = \Pi_a\left(W^{a} + \pmb\theta_0 \in A\right) \geq e^{-\frac{1}{2}\lVert \pmb\theta_0\rVert^{2}_{\mathcal{H}_a}}\Pi_a\left(W^{a} \in A\right).
\end{align*}
\end{lemma}

\begin{proof}
Since, $\pmb{\theta}_0(\cdot) \in \mathcal{H}_a$, the Gaussian measures induced by $W^{a}$ and $W^{a} + \pmb\theta_0$ on $\mathcal{C}[0, 1]^{d}$ are mutually absolutely continuous and the Radon-Nikodym derivative is given by the Cameron-Martin formula~\citep{cameron-martin}:
$$
\frac{d\Pi_a^{\pmb\theta_0}}{d\Pi_a}(w) = e^{\langle w, \pmb\theta_0\rangle_{\mathcal{H}_a} - \frac{1}{2}\lVert \pmb\theta_0\rVert_{\mathcal{H}_a}^{2}},
$$
where $\Pi_a^{\pmb\theta_0}$ denotes the law of $W^{a} + \pmb\theta_0$. Then, for any measurable $A$:
$$
\Pi_a^{\pmb\theta_0}(A) = \int \text{I}_A(w)e^{\langle w, \pmb\theta_0\rangle_{\mathcal{H}_a} - \frac{1}{2}\lVert \pmb\theta_0\rVert_{\mathcal{H}_a}^2}d\Pi_a(w),
$$
where $\text{I}_A(w) = 1$ if $w\in A$ and $\text{I}_A(w) = 0$ otherwise. Applying Jensen's inequality and using $\Pi_a(A) \leq 1$, yields the bound:
\begin{align*}
    \Pi_a^{\pmb\theta_0}(A) \geq e^{-\frac{1}{2}\lVert \pmb\theta_0\rVert^{2}_{\mathcal{H}_a}}\Pi_a(A).
\end{align*}
\end{proof}

\section{Posterior Contraction Rate for Squared Exponential Kernel}\label{app:SE-contraction}

\begin{lemma}
\label{lemma:SE-contraction-rate}
Consider the nonparametric regression setting in~\eqref{eq:regression-setting} and a conditional Gaussian process prior $\pmb\theta \mid a \sim \mathrm{GP}(0, k_a(\cdot, \cdot))$ having the squared exponential kernel:
$$k_a(x, x') = e^{-a^{2}\lVert x-x'\rVert_{2}^2}, \quad a>0.
$$ 
The prior for $a$ satisfies, $g(a) \geq A_1 a^{p}e^{-B_1 a^{d} \log^{q}a}$. For an arbitrary $a>0$, the true function $\pmb\theta_0(\cdot)\in \mathcal{H}_u$ for all $u\in I(a):= [a/2, 2a]$ and $\sup_{u\in I(a)}\lVert \pmb\theta_0\rVert_{\mathcal{H}_u}^{2} \leq R^{2}_{I(a)} < \infty$, where $\mathcal{H}_a$ is the corresponding prior RKHS. Then, the $\alpha$-posterior contracts at rate:
\begin{align*}
    \epsilon_t = \mathcal{O}\left(t^{-\frac{1}{2}}\log^{\frac{d+1}{2}}t\right).
\end{align*}
\end{lemma}

\begin{proof}
Following~\cite{bhattacharya2019bayesian} it suffices to verify a prior mass bound for the $2$-R\'{e}nyi neighborhood as defined in Assumption~\ref{ass:prior}:
\begin{align*}
    B(\pmb\theta_0, \epsilon_t) := \left\{\pmb\theta \in \Theta\;:\;D_2^{(t)}(\pmb\theta_0, \pmb\theta) \leq \frac{D\alpha t \epsilon_t^2}{4}\right\}.
\end{align*}
Note that, $D_2^{(t)}(\pmb\theta_0, \pmb\theta) = t\sigma^{-2}\lVert \pmb\theta - \pmb\theta_0\rVert_{2}^{2}$ from~\eqref{eq:regression-setting} with known variance $\sigma^{2}\in \mathbb{R}^{+}$. Hence:
\begin{align*}
B(\pmb\theta_0, \epsilon_t) = \left\{\pmb\theta\in \Theta\;:\;\lVert \pmb\theta - \pmb\theta_0\rVert_{2} \leq \frac{\sigma}{2}\sqrt{D\alpha}\epsilon_t\right\}.
\end{align*}
We have the containment:
\begin{align}
\label{eq:SE-contraction-1}
\left\{\lVert \pmb\theta - \pmb\theta_0\rVert_{\infty} \leq C\epsilon_t\right\} \subset B(\pmb\theta_0, \epsilon_t),
\end{align}
where $C:= \sigma \sqrt{D\alpha}/2$. Writing $\Pi(\cdot)$ for the unconditional prior and $\Pi_{u}(\cdot)$ for the conditional law $\pmb\theta\mid a=u$. For any measurable event $E$:
\begin{align*}
\Pi(E) = \int_{0}^{\infty}\Pi_u(E) g(u)du \geq \int_{I(a)}\Pi_u(E)g(u)du,
\end{align*}
where $I(a):= [a/2, 2a]$. Using the hyperprior lower bound~\citep{bhattacharya2019bayesian}, $g(a) \geq A_1 a^p e^{-B_1a^d\log^{q}a}$, for all $u\in I(a)$:
\begin{align*}
g(u) \geq A_1\left(\frac{a}{2}\right)^{p}e^{-B_1(2a)^d\log^{q}(2a)} =: g_0(a) > 0.
\end{align*}
Hence:
\begin{align}
\label{eq:SE-contraction-2}
\Pi(E) \geq |I(a)| g_0(a)\inf_{u\in I(a)}\Pi_u(E).
\end{align}
We apply~\eqref{eq:SE-contraction-2} with $E:= \{\lVert \pmb\theta-\pmb\theta_0\rVert_{\infty} \leq C\epsilon_t\}$, so it suffices to lower bound $\inf_{u\in I(a)}\Pi_u(\lVert \pmb\theta-\pmb\theta_0\rVert_{\infty} \leq C\epsilon_t)$.

Fix $u\in I(a)$ and let $W^{u}\sim \mathrm{GP}(0, k_u(\cdot, \cdot))$. By Lemma~\ref{lemma:cameron-martin-formula}, if $\pmb\theta_0(\cdot) \in \mathcal{H}_u$, where $\mathcal{H}_u$ denotes the RKHS associated with the kernel of $W^{u}$, then for the sup-norm ball $A:= \{f: \lVert f \rVert_{\infty} \leq C\epsilon_t\}$:
\begin{align*}
\Pi_u\left(\lVert \pmb\theta-\pmb\theta_0\rVert_{\infty}\leq C\epsilon_t\right) \geq e^{-\frac{1}{2}\lVert\pmb\theta_0\rVert^{2}_{\mathcal{H}_u}}\Pi_u\left(\lVert W^{u}\rVert_{\infty}\leq C\epsilon_t\right).
\end{align*}
Assuming $\pmb\theta_0(\cdot) \in \mathcal{H}_u$ for all $u\in I(a)$ and using $\sup_{u\in I(a)}\lVert \pmb{\theta}_0\rVert_{\mathcal{H}_u}^{2} \leq R_{I(a)}^{2}< \infty$, we have:
\begin{align}
\label{eq:SE-contraction-3}
\inf_{u\in I(a)}\Pi_u\left(\lVert \pmb\theta-\pmb\theta_0\rVert_{\infty} \leq C\epsilon_t\right) \geq e^{-\frac{R^{2}_{I(a)}}{2}}\inf_{u\in I(a)}\Pi_u\left(\lVert W^{u}\rVert_{\infty} \leq C\epsilon_t\right).
\end{align}

For the squared exponential kernel on $[0, 1]^{d}$, there exist constants $C_1 > 0$ and $\epsilon_0\in (0, 1)$ such that for all $\epsilon \in (0, \epsilon_0)$~\citep{vanderVaart2009}:
\begin{align}
\label{eq:SE-contraction-4}
-\log \Pi_u\left(\lVert W^{u}\rVert_{\infty} \leq \epsilon\right) \leq C_1 (u \vee 1) \log^{d+1}\left(\frac{u \vee 1}{\epsilon}\right),
\end{align}
where $u\vee 1:= \max\{u, 1\}$. For $u\in I(a)$, we have $(u\vee 1) \leq (2a\vee 1)$, so applying \eqref{eq:SE-contraction-4} with $\epsilon = C\epsilon_t$ yields:
\begin{align*}
-\log \Pi_u\left(\lVert W^{\infty}\rVert_{\infty} \leq C\epsilon_t\right) \leq C_1 (2a\vee 1)^{d} \log^{d+1}\left(\frac{2a\vee 1}{C\epsilon_t}\right).
\end{align*}
Therefore:
\begin{align}
\label{eq:SE-contraction-5}
\inf_{u\in I(a)} \Pi_u\left(\lVert W^{u}\rVert_{\infty} \leq C\epsilon_t\right) \geq e^{-C_1 (2a\vee 1)^{d} \log^{d+1}\left(\frac{2a\vee 1}{C\epsilon_t}\right)}.
\end{align}
Combining~\eqref{eq:SE-contraction-3} and~\eqref{eq:SE-contraction-5} gives:
\begin{align}
\label{eq:SE-contraction-6}
\inf_{u\in I(a)}\Pi_u\left(\lVert \pmb\theta-\pmb\theta_0\rVert_{\infty} \leq C\epsilon_t\right) \geq e^{-\frac{R^{2}_{I(a)}}{2} - C_1 (2a\vee 1)^{d} \log^{d+1}\left(\frac{2a\vee 1}{C\epsilon_t}\right)}.
\end{align}
Plugging~\eqref{eq:SE-contraction-6} into~\eqref{eq:SE-contraction-2} yields:
\begin{align*}
\Pi\left(\lVert \pmb\theta - \pmb\theta_0\rVert_{\infty}\leq C\epsilon_t\right) \geq |I(a)| g_0(a) e^{-\frac{R^{2}_{I(a)}}{2} - C_1(2a\vee 1)^{d}\log^{d+1}\left(\frac{2a\vee 1}{C\epsilon_t}\right)},
\end{align*}
hence:
\begin{align}
\label{eq:SE-contraction-7}
-\log \Pi\left(\lVert \pmb\theta - \pmb\theta_0\rVert_{\infty} \leq C\epsilon_t\right) \leq \frac{R^{2}_{I(a)}}{2} - \log(|I(a)| g_0(a)) + C_1(2a\vee 1)^{d}\log^{d+1}\left(\frac{2a\vee 1}{C\epsilon_t}\right).
\end{align}
By~\eqref{eq:SE-contraction-1}, the same upper bound holds for $-\log \Pi(B(\pmb\theta_0, \epsilon_t))$.

Now, define the sequence:
\begin{align}
\label{eq:SE-contraction-8}
\epsilon_t^{2}(a) := \frac{2}{t}\left[\frac{R^{2}_{I(a)}}{2} - \log(|I(a)|g_0(a))\right] + \frac{2C_1(2a\vee 1)^{d}}{t}\log^{d+1}\left(\frac{(2a\vee 1)\sqrt{t}}{C}\right).
\end{align}
Substituting $\epsilon_t^{2}(a)$ given by~\eqref{eq:SE-contraction-8} in~\eqref{eq:SE-contraction-7} and using $\log\left(\frac{2a\vee 1}{C\epsilon_t(a)}\right) \leq \log\left(\frac{(2a\vee 1)\sqrt{t}}{C}\right)$ for $t$ large gives prior mass bound:
\begin{align*}
-\log \Pi(B(\pmb\theta_0, \epsilon_t(a))) \lesssim t\epsilon_t^2(a).
\end{align*}
Moreover, since $\log\left(\frac{(2a\vee 1)\sqrt{t}}{C}\right) \asymp \frac{1}{2}\log t$ for fixed $a$:
\begin{align*}
    \epsilon_t^{2}(a) \asymp \frac{(a\vee 1)^d \log^{d+1}t}{t},
\end{align*}
is the posterior contraction rate following~\cite{bhattacharya2019bayesian}[Corollary 3.5].

Let: 
\begin{align*}
    \mathcal{A} := \left\{a>0\;:\;\pmb\theta_0(\cdot)\in \mathcal{H}_u\text{ for all }u\in I(a), R_{I(a)}:=\sup_{u\in I(a)}\lVert\pmb\theta_0\rVert_{\mathcal{H}_{u}} < \infty\right\},
\end{align*}
and define $\epsilon_t^2:= \inf_{a\in \mathcal{A}}\epsilon_t^{2}(a)$, where $\epsilon_t^{2}(a)$ is given in~\eqref{eq:SE-contraction-8}. Pick $a_t\in \mathcal{A}$ such that $\epsilon_t^{2}(a) \leq \epsilon_t^{2} + t^{-1}$. Then, for this $\epsilon_t^{2}$ the prior mass condition, $-\log \Pi(\lVert \pmb\theta-\pmb\theta_0\rVert_{\infty} < C\epsilon_t(a)) \lesssim t\epsilon_t^{2}(a) \lesssim t\epsilon_t^{2}$, holds. In particular, for some finite $a_{\star}\in \mathcal{A}$, we have:
\begin{align}
\label{eq:SE-contraction-9}
\epsilon_t^{2} \lesssim \frac{(a_{\star}\vee 1)^d\log^{d+1}t}{t} \implies \epsilon_t = \mathcal{O}\left(t^{-\frac{1}{2}}\log^{\frac{d+1}{2}}t\right).
\end{align}
\end{proof}

\section{Posterior Contraction Rate for Mat\'{e}rn-$\nu$ Kernel}
\label{app:matern-contraction-rate}

\begin{lemma}
\label{lemma:matern-contraction-rate}
Consider the nonparametric regression setting in~\eqref{eq:regression-setting} and a conditional Gaussian process prior $\pmb\theta\mid a\sim \mathrm{GP}(0, k_a^{(\nu)}(\cdot, \cdot))$ having the Mat\'{e}rn-$\nu$ kernel: 
$$
k_a^{(\nu)}(x, x') = 2^{1-\nu}\Gamma^{-1}(\nu)(\sqrt{2\nu}a\lVert x-x'\rVert_{2})^{\nu} K_{\nu}(\sqrt{2\nu}a\lVert x-x'\rVert_2), \quad a>0, \nu>\frac{d}{2}.
$$
The prior for $a$ satisfies, $g(a) \geq A_1 a^{p}e^{-B_1 a^{d}\log^{q}a}$. For constants $0<a_L < a_U < \infty$, the true function $\pmb\theta_0(\cdot) \in \mathcal{H}_a^{(\nu)}$ for all $a\in [a_L, a_U]$ and $\sup_{a\in [a_L, a_U]}\lVert \pmb\theta_0\rVert_{\mathcal{H}_a^{(\nu)}}^{2} \leq R^{2} < \infty$, where $\mathcal{H}_a^{(\nu)}$ is the corresponding prior RKHS. Then, the $\alpha$-posterior contracts at rate:
\begin{align*}
\epsilon_t = t^{-\frac{\nu}{2\nu+d}}\log^{\frac{q}{2+\alpha^{-1}d}}t.
\end{align*}
\end{lemma}

\begin{proof}
We follow the same road map as in the proof of Lemma~\ref{lemma:SE-contraction-rate} above. Writing $\Pi(\cdot)$ for the unconditional prior and $\Pi_a(\cdot)$ for the conditional law $\pmb\theta \mid a$. Let $W^{a} \sim \mathrm{GP}(0, k^{(\nu)}_a(\cdot, \cdot))$ under $\Pi_a(\cdot)$. 

Also, let $D_2^{(t)}(\pmb\theta_0, \pmb\theta) = \sum_{s=1}^{t}D_2^s(\pmb\theta_0, \pmb\theta)$ be the $2$-R\'{e}nyi divergence. By the local quadratic R\'{e}nyi-Bernstein condition~\citep{ghosal2000convergence,kleijn2006misspecification,bhattacharya2019bayesian}, for all $\pmb\theta$ such that $\lVert \pmb\theta-\pmb\theta_0\rVert_{\infty} \leq c_0$:
\begin{align*}
D_2^s(\pmb\theta_0, \pmb\theta) \leq c_L (\pmb\theta(x_s) - \pmb\theta_0(x_s))^2,\quad s = 1, 2, \ldots, t,
\end{align*}
where $c_0, c_L > 0$ are constants. Summing and using $|\pmb\theta(x_s) - \pmb\theta_0(x_s)| \leq \lVert \pmb\theta-\pmb\theta_0\rVert_{\infty}$ yields:
\begin{align*}
D_2^{(t)}(\pmb\theta_0, \pmb\theta) \leq c_L\sum_{s=1}^{t}(\pmb\theta(x_s) - \pmb\theta_0(x_s))^2 \leq c_L t\lVert\pmb\theta-\pmb\theta_0\rVert_{\infty}^{2}.
\end{align*}
Therefore, choosing $C'' := \sqrt{D\alpha / 4 c_L}$, we have the set inclusion for all sufficiently small $\epsilon_t$ (so that $C''\epsilon_t \leq c_0$):
\begin{align*}
\left\{\lVert \pmb\theta-\pmb\theta_0\rVert_{\infty} \leq C''\epsilon_t\right\} \subset B(\pmb\theta_0, \epsilon_t):= \left\{\pmb\theta\in \Theta\;:\;D_2^{(t)}(\pmb\theta_0, \pmb\theta) \leq \frac{D\alpha t \epsilon_t^{2}}{4}\right\}.
\end{align*}
Hence:
\begin{align*}
\Pi\left(B(\pmb\theta_0, \epsilon_t)\right) \geq \Pi\left(\lVert \pmb\theta-\pmb\theta_0\rVert_{\infty} < C'' \epsilon_t\right).
\end{align*}

For any fixed $a\in [a_L, a_U]$, since $\pmb\theta_0(\cdot) \in \mathcal{H}_a^{(\nu)}$, Lemma~\ref{lemma:cameron-martin-formula} with $A:= \{f\;:\; \lVert f\rVert_{\infty} \leq C''\epsilon_t\}$ yields:
\begin{align}
\label{eq:matern-contraction-1}
\Pi_a\left(\lVert \pmb\theta-\pmb\theta_0\rVert_{\infty} \leq C''\epsilon_t\right) \geq e^{-\frac{1}{2}\lVert \pmb\theta_0\rVert_{\mathcal{H}_a^{(\nu)}}^{2}}\Pi_a\left(\lVert W^{a}\rVert_{\infty} \leq C''\epsilon_t\right).
\end{align}
As $\sup_{a\in [a_L, a_U]}\lVert \pmb\theta_0\rVert_{\mathcal{H}_a^{(\nu)}}^{2} \leq R^{2} < \infty$, for all $a\in [a_L, a_U]$ we have $e^{-\frac{1}{2}\lVert \pmb\theta_0\rVert_{\mathcal{H}_a^{(\nu)}}^{2}} \geq e^{-\frac{R^2}{2}}$.

For Mat\'{e}rn-$\nu$ kernel on $[0, 1]^{d}$ with $\nu > \frac{d}{2}$, there exist $C_1 > 0$ and $\epsilon_0 \in (0, 1)$ such that for all $a\in [a_L, a_U]$ and all $0< \epsilon \leq \epsilon_0$~\citep{vdv-vz-2008,vdv-vz-2011}:
\begin{align*}
-\log \Pi_a\left(\lVert W^{a}\rVert_{\infty} \leq \epsilon\right) \leq C_1 a^{d}\epsilon^{-\frac{d}{\nu}}.
\end{align*}
Taking $\epsilon = C'' \epsilon_t$ and using $a^{d} \leq a_U^{d}$ gives:
\begin{align*}
-\log \Pi_a\left(\lVert W^{a}\rVert_{\infty} \leq C''\epsilon_t\right) \leq C_1 a_U^{d} (C''\epsilon_t)^{-\frac{d}{\nu}}.
\end{align*}
Combining with~\eqref{eq:matern-contraction-1} yields, uniformly over $a\in [a_L, a_U]$:
\begin{align}
\label{eq:matern-contraction-2}
\Pi_a\left(\lVert \pmb\theta-\pmb\theta_0\rVert_{\infty} \leq C''\epsilon_t\right) \geq e^{-\frac{R^2}{2} - C_1 a_U^{d} (C''\epsilon_t)^{-\frac{d}{\nu}}}.
\end{align}

Let $I(a):= [a_L, a_U]$. Then, for any measurable event $E$:
\begin{align*}
\Pi(E) = \int_{0}^{\infty}\Pi_a(E) g(a)da \geq \int_{I(a)} \Pi_a(E) g(a) da. 
\end{align*}
Since, $g$ is bounded below on the compact set $I(a)$ as $g(a) \geq A_1 a^{p}e^{-B_1 a^{d}\log^{q}a}$, there exist $g_{I(a)} > 0$ such that $\inf_{a\in I(a)} g(a) \geq g_{I(a)}$. Hence:
\begin{align}
\label{eq:matern-contraction-3}
\Pi(E) \geq |I(a)| g_{I(a)} \inf_{a\in I(a)}\Pi_a(E).
\end{align}
From~\eqref{eq:matern-contraction-3}, it suffices to lower bound $\inf_{a\in I(a)}\Pi_a(E)$. Now, take $E:= \{\lVert \pmb\theta-\pmb\theta_0\rVert_{\infty} \leq C''\epsilon_t\}$. By~\eqref{eq:matern-contraction-2}, the lower bound is uniform in $a\in I(a)$, so:
\begin{align*}
\inf_{a\in I(a)} \Pi_a(E) \geq e^{-\frac{R^2}{2} - C_1 a_U^{d}(C''\epsilon_t)^{-\frac{d}{\nu}}}.
\end{align*}
Substituting into~\eqref{eq:matern-contraction-3} gives:
\begin{align}
\label{eq:matern-contraction-4}
\Pi\left(\lVert \pmb\theta - \pmb\theta_0\rVert_{\infty} \leq C''\epsilon_t\right) \geq |I(a)| g_{I(a)}e^{-\frac{R^2}{2} - C_1 a_U^{d}(C'' \epsilon_t)^{-\frac{d}{\nu}}}.
\end{align}
\eqref{eq:matern-contraction-4} implies:
\begin{align*}
-\log \Pi\left(\lVert \pmb\theta-\pmb\theta_0\rVert_{\infty} \leq C''\epsilon_t\right) \leq C_2 + C_3\epsilon_t^{-\frac{d}{\nu}},
\end{align*}
where $C_2 := -\log (|I(a)|g_{I(a)}) + \frac{R^2}{2}$ and $C_3:= C_1 a_U^{d}(C'')^{-\frac{d}{\nu}}$.

Therefore, following~\cite{bhattacharya2019bayesian}[Corollary 3.5] the posterior contraction rate with which the prior mass condition, $-\log \Pi(B(\pmb\theta_0, \epsilon_t)) \lesssim t\epsilon_t^2$, holds is:
\begin{align}
\label{eq:matern-contraction-5}
\epsilon_t = t^{-\frac{\nu}{2\nu+d}}\log^{\frac{q}{2 + \alpha^{-1}d}}t.
\end{align}
\end{proof}

\section{Posterior Contraction Rate for Rational Quadratic Kernel}
\label{app:rq-contraction-rate}

\begin{lemma}
\label{lemma:rq-contraction-rate}
Consider the nonparametric regression setting in~\eqref{eq:regression-setting} and a conditional Gaussian process prior $\pmb\theta\mid a\sim \mathrm{GP}(0, k_{\mathrm{RQ}, a}(\cdot, \cdot))$ having the rational quadratic kernel:
\begin{align*}
k_{\mathrm{RQ}, a}(x, x') = \left(1 + \frac{a^{2}\lVert x - x'\rVert_{2}^{2}}{2\nu}\right)^{-\nu},\quad a>0, \nu>\frac{d}{2}.
\end{align*}
The prior for $a$ satisfies, $g(a) \geq A_1 a^{p}e^{-B_1 a^{d}\log^{q}a}$. For constants $0<a_L < a_U < \infty$, the true function $\pmb\theta_0(\cdot) \in \mathcal{H}_a$ for all $a\in [a_L, a_U]$ and $\sup_{a\in [a_L, a_U]}\lVert \pmb\theta_0\rVert_{\mathcal{H}_a}^{2} \leq R^{2} < \infty$, where $\mathcal{H}_a$ is the corresponding prior RKHS. Then, the $\alpha$-posterior contracts at rate:
\begin{align*}
\epsilon_t \asymp t^{-\frac{\nu}{2\nu+d}}.
\end{align*}
\end{lemma}

\begin{proof}
The inverse scale-length parameterization $(a:=\ell^{-1})$ of the rational quadratic kernel is:
\begin{align*}
k_{\mathrm{RQ}, a}(x, x') = \left(1 + \frac{a^{2}\lVert x - x'\rVert_{2}^{2}}{2\nu}\right)^{-\nu}.
\end{align*}

\medskip
\noindent\textit{Centered sup-norm small ball for the rational quadratic kernel}. Let $W^{a} \sim \mathrm{GP}(0, k_{\mathrm{RQ}, a}(\cdot, \cdot))$ on $[0, 1]^{d}$, where $\nu > d/2$. Fix any compact interval $[a_L, a_U]$ for constants $0<a_L<a_U<\infty$. Then, there exist constants $C>0$ and $\epsilon_0 \in (0, 1)$ such that for all $a\in [a_L, a_U]$ and all $\epsilon\in (0, \epsilon_0)$:
\begin{align}
\label{eq:rq-contraction-1}
-\log \mathbb{P}\left(\lVert W^{a}\rVert_{\infty} \leq \epsilon\right) \leq C (a\vee 1)^{d}\epsilon^{-\frac{d}{\nu}},
\end{align}
where $a\vee 1:=\max\{a, 1\}$.

We first establish~\eqref{eq:rq-contraction-1}. Since, $k_{\mathrm{RQ}, a}(\cdot,\cdot)$ is stationary, by Bochner's theorem it admits a spectral density $f_a(\omega)$ such that:
\begin{align*}
k_{\mathrm{RQ}, a}(h) = \int_{\mathbb{R}^{d}}e^{i\omega\cdot h}f_a(\omega)d\omega,\quad h\in \mathbb{R}^{d}.
\end{align*}
For the rational quadratic kernel, the spectral density has the closed form~\citep{Rasmussen-Williams-2005}:
\begin{align}
\label{eq:rq-contraction-2}
f_{a}(\omega) = C_{d, \nu}a^{-d}\left(1 + \frac{\lVert \omega \rVert_{2}^{2}}{2\nu a^{2}}\right)^{-(\nu + \frac{d}{2})},
\end{align}
for a constant $C_{d, \nu} > 0$. In particular, there exist constants $0<c_1 \leq c_2<\infty$ such that for all $a\in [a_L, a_U]$ and all $\omega \in \mathbb{R}^{d}$:
\begin{align}
\label{eq:rq-contraction-3}
c_1a^{-d}\left(a^{2} + \lVert \omega \rVert_{2}^{2}\right)^{-(\nu + \frac{d}{2})} \leq f_a(\omega) \leq c_2a^{-d}\left(a^{2} + \lVert \omega \rVert_{2}^{2}\right)^{-(\nu + \frac{d}{2})}.
\end{align}
Now, let $\mathcal{H}_a$ be the RKHS of $\mathrm{GP}(0, k_{\mathrm{RQ}, a}(\cdot, \cdot))$ defined on $[0, 1]^{d}$. For stationary GPs, the RKHS norm is characterized by the spectral density~\citep{vdv-vz-2008}, i.e., up to a constant:
\begin{align}
\label{eq:rq-contraction-4}
\lVert h\rVert_{\mathcal{H}_a}^{2} \asymp \int_{\mathbb{R}^{d}}\frac{|\widehat{h}(\omega)|^2}{f_a(\omega)}d\omega.
\end{align}
Combining~\eqref{eq:rq-contraction-3} and~\eqref{eq:rq-contraction-4}, on $a\in [a_L, a_U]$:
\begin{align}
\label{eq:rq-contraction-5}
\lVert h \rVert_{\mathcal{H}_a}^{2} \asymp a^{d}\int_{\mathbb{R}^{d}}\left(a^2 + \lVert \omega\rVert_2^{2}\right)^{\nu + \frac{d}{2}}|\widehat{h}(\omega)|^{2}d\omega.
\end{align}

Consider, $\mathcal{H}_{a, 1}:= \{h\in \mathcal{H}_a\;:\; \lVert h\rVert_{\mathcal{H}_a} \leq 1\}$. We will show that there exist constants $C'>0, \eta_0>0$ such that for all $a\in [a_L, a_U]$ and all $\eta\in (0, \eta_0)$:
\begin{align}
\label{eq:rq-contraction-6}
\log N(\eta, \mathcal{H}_{a, 1}, \lVert \cdot \rVert_{\infty}) \leq C'(a\vee 1)^{d}\eta^{-\frac{d}{\nu}},
\end{align}
where $N(\eta, \mathcal{H}_{a, 1}, \lVert \cdot \rVert_{\infty})$ denotes the minimum number of balls of radius $\eta>0$ needed to cover the unit ball $\mathcal{H}_{a, 1}$ of the RKHS. To avoid boundary artifacts, embed $[0, 1]^{d}$ into the torus $\mathbb{T}^{d}$ via periodic extension. Let $\{e_k(x) = e^{2\pi i k\cdot x}\}_{k\in \mathbb{Z}^{d}}$ be the Fourier basis on $\mathbb{T}^{d}$, and write $h(x) = \sum_{k\in \mathbb Z^{d}}\widehat{h}_k e_k(x)$. From~\eqref{eq:rq-contraction-5} and $\lVert h\rVert_{\mathcal{H}_a} \leq 1$, we obtain:
\begin{align}
\label{eq:rq-contraction-7}
\sum_{k\in \mathbb Z^{d}}a^{d}\left(a^{2} + \lVert k\rVert_{2}^{2}\right)^{\nu + \frac{d}{2}}|\widehat{h}_k|^{2} \leq C_0,
\end{align}
where $C_0$ is a constant depending on $(d, \nu, a_L, a_U)$. Let $P_m h:= \sum_{\lVert k \rVert_{\infty}\leq m a}\widehat{h}_k e_k$. Then:
\begin{align}
\label{eq:rq-contraction-8}
\lVert h - P_m h\rVert_{\infty} \leq \sum_{\lVert k \rVert_{\infty} > ma} |\widehat{h}_k|.
\end{align}
Apply Cauchy-Schwarz inequality with weights, $w_k := a^{d}(a^{2} + \lVert k \rVert_{2}^{2})^{\nu + d/2}$, to obtain:
\begin{align}
\label{eq:rq-contraction-9}
\sum_{\lVert k \rVert_{\infty} > ma}|\widehat{h}_k| \leq \left(\sum_{\lVert k \rVert_{\infty} > ma}w_k|\widehat{h}_k|^2\right)^{\frac{1}{2}} \left(\sum_{\lVert k \rVert_{\infty}>ma} w_k^{-1}\right)^{\frac{1}{2}}.
\end{align}
The first factor in~\eqref{eq:rq-contraction-9} is less than equal to $C_0^{1/2}$ by~\eqref{eq:rq-contraction-7}. For the second factor in~\eqref{eq:rq-contraction-9}:
\begin{align}
\label{eq:rq-contraction-10}
\sum_{\lVert k \rVert_{\infty} > ma}w_k^{-1} = \sum_{\lVert k \rVert_{\infty} > ma}a^{-d}\left(a^{2} + \lVert k \rVert_{2}^{2}\right)^{-(\nu + \frac{d}{2})} \leq a^{-d} \sum_{\lVert k\rVert_{\infty}> ma} \lVert k \rVert_{2}^{-2\nu - d},
\end{align}
as $\lVert k \rVert_{2} \gtrsim a$ and $a^{2} + \lVert k \rVert_{2}^{2} \asymp \lVert k  \rVert_2^2$. Approximating the sum by an integral and using Euclidean norm $\lVert \cdot \rVert_2$, we get:
\begin{align}
\label{eq:rq-contraction-11}
\sum_{\lVert k \rVert_{\infty} > ma}\lVert k \rVert_{2}^{-2\nu-d} \lesssim \int_{\lVert u \rVert_{2} > ma}\lVert u \rVert_2^{-2\nu - d}du \asymp (ma)^{-2\nu}.
\end{align}
Plugging~\eqref{eq:rq-contraction-11} into~\eqref{eq:rq-contraction-10} gives:
\begin{align}
\label{eq:rq-contraction-12}
\sum_{\lVert k \rVert_{\infty} > ma}w_k^{-1}\lesssim a^{-d}(ma)^{-2\nu}.
\end{align}
Then, from~\eqref{eq:rq-contraction-9}:
\begin{align}
\label{eq:rq-contraction-13}
\lVert h - P_m h\rVert_{\infty} \lesssim C_0^{\frac{1}{2}}\left(a^{-d}(ma)^{-2\nu}\right)^{\frac{1}{2}} \asymp m^{-\nu}a^{-(\nu + \frac{d}{2})}.
\end{align}
Since, $a\in [a_L, a_U]$, $a^{-(\nu+d/2)}$ is uniformly bounded by a constant depending only on the slab, $[a_L, a_U]$. Hence:
\begin{align}
\label{eq:rq-contraction-14}
\sup_{a\in [a_L, a_U]}\sup_{h\in \mathcal{H}_{a, 1}}\lVert h - P_m h\rVert_{\infty} \leq C_2 m^{-\nu},
\end{align}
for some constant $C_2$ depending on $(d, \nu, a_L, a_U)$. Therefore, choosing $m\asymp \eta^{-1/\nu}$ guarantees $\lVert h - P_mh\rVert_{\infty} \leq \eta/2$. The truncated set $P_m(\mathcal{H}_{a, 1})$ lives in a linear space of dimension, $D_m:= \# \{k\in \mathbb{Z}^{d}\;:\; \lVert k \rVert_{\infty} \leq ma\} \asymp (ma)^{d}$. A standard volumetric argument gives that the $\lVert \cdot \rVert_{\infty}$-covering number of a bounded subset of a $D_m$-dimensional normed space satisfies:
\begin{align}
\label{eq:rq-contraction-15}
\log N\left(\frac{\eta}{2}, P_m(\mathcal{H}_{a, 1}), \lVert \cdot \rVert_{\infty}\right) \lesssim D_m \log\left(\frac{1}{\eta}\right).
\end{align}
Combining $D_m\asymp (ma)^{d}$ and $m\asymp \eta^{-1/\nu}$, we get:
\begin{align}
\label{eq:rq-contraction-16}
\log N\left(\frac{\eta}{2}, P_m(\mathcal{H}_{a, 1}), \lVert \cdot \rVert_{\infty}\right) \lesssim (ma)^{d}\log\left(\frac{1}{\eta}\right) \asymp a^{d}\eta^{-\frac{d}{\nu}}\log\left(\frac{1}{\eta}\right).
\end{align}
Thus, for small $\eta$, we establish~\eqref{eq:rq-contraction-6}:
\begin{align*}
    \log N(\eta, \mathcal{H}_{a, 1}, \lVert \cdot \rVert_{\infty}) \leq C'(a \vee 1)^{d} \eta^{-\frac{d}{\nu}}.
\end{align*}

Let $\phi_a(\epsilon):= -\log \mathbb{P}(\lVert W^{a}\rVert_{\infty} \leq \epsilon)$. A fundamental result in the small deviation literature states that for Gaussian measures on Banach spaces, small-ball probabilities are controlled by metric entropy of the RKHS unit ball~\citep{vanZantenVanDerVaart2008IMS}. In particular, there exist constants $c_3, c_4 > 0$ such that for sufficiently small $\epsilon$:
\begin{align}
\label{eq:rq-contraction-17}
\phi_a(\epsilon) \leq c_3 \log N(c_4 \epsilon, \mathcal{H}_{a, 1}, \lVert \cdot \rVert_{\infty}).
\end{align}
Using~\eqref{eq:rq-contraction-17} and~\eqref{eq:rq-contraction-6}, for small $\epsilon$, we have~\eqref{eq:rq-contraction-1}:
\begin{align*}
\phi_a(\epsilon) = -\log \mathbb{P}(\lVert W^{a}\rVert_{\infty} \leq \epsilon) \leq c_3 C'(a\vee 1)^{d}(c_4\epsilon)^{-\frac{d}{\nu}} = C(a\vee 1)^{d}\epsilon^{-\frac{d}{\nu}}.
\end{align*}

\medskip
\noindent\textit{Rational quadratic kernel contraction rate}. We proceed in a similar fashion as in the proof of Lemma~\ref{lemma:matern-contraction-rate}. Invoking the local quadratic R\'{e}nyi-Bernstein condition yields:
\begin{align}
\label{eq:rq-contraction-rate-1}
\Pi\left(B(\pmb\theta_0, \epsilon)\right) \geq \Pi\left(\lVert \pmb\theta - \pmb\theta_0\rVert_{\infty} \leq r\right),
\end{align}
for $\epsilon >0$, where $B(\pmb\theta_0, \epsilon) := \{\pmb\theta \in \Theta\;:\; D_2^{(t)}(\pmb\theta_0, \pmb\theta) \leq D\alpha t\epsilon^2/4\}$, $\Pi(\cdot)$ denotes the unconditional prior, and $\Pi_a(\cdot)$ denotes the conditional law $\pmb\theta \mid a$. Note that, $r:= \sqrt{D\alpha /4c_L}\epsilon \leq c_0$, for constants $c_0, c_L>0$.

Let $E_r:= \{\lVert \pmb\theta - \pmb\theta_0\rVert_{\infty}\leq r\}$. With $\Pi(\cdot) = \int \Pi_a(\cdot)g(a) da$:
\begin{align}
\label{eq:rq-contraction-rate-2}
\Pi(E_r) = \int_{0}^{\infty}\Pi_a(E_r) g(a)da \geq \int_{a_L}^{a_U}\Pi_a(E_r)g(a)da \geq (a_U-a_L)\underline{g}\inf_{a\in [a_L, a_U]}\Pi_a(E_r),
\end{align}
where $\underline{g} > 0$ such that $\inf_{a\in [a_L, a_U]}g(a) \geq \underline{g}$. From~\eqref{eq:rq-contraction-rate-2}, it suffices to lower bound $\inf_{a\in [a_L, a_U]}\Pi_a(E_r)$.

Fix any $a\in [a_L, a_U]$. Under $\Pi_a(\cdot)$, $\pmb\theta = W^{a}$ with $W^{a} \sim \mathrm{GP}(0, k_{\mathrm{RQ}, a}(\cdot, \cdot))$. Thus, $\Pi_a(E_r) = \Pi_a(\lVert W^{a} - \pmb\theta_0\rVert_{\infty} \leq r)$. As $\pmb\theta_0(\cdot) \in \mathcal{H}_a$ (prior RKHS), Lemma~\ref{lemma:cameron-martin-formula} yields:
\begin{align}
\label{eq:rq-contraction-rate}
\Pi_a\left(\lVert W^{a} - \pmb\theta_0\rVert_{\infty} \leq r\right) \geq e^{-\frac{1}{2}\lVert \pmb\theta_0\rVert^{2}_{\mathcal{H}_a}}\Pi_a\left(\lVert W^{a}\rVert_{\infty} \leq r\right).
\end{align}
Taking infimum over $a\in [a_L, a_U]$ and using $\sup_{a\in [a_L, a_U]}\lVert \pmb\theta_0\rVert^{2}_{\mathcal{H}_a} \leq R^{2} < \infty$:
\begin{align}
\label{eq:rq-contraction-rate-4}
\inf_{a\in [a_L, a_U]}\Pi_{a}(E_r) \geq e^{-\frac{R^2}{2}}\Pi_{a}\left(\lVert W^{a}\rVert_{\infty} \leq r\right).
\end{align}
Using~\eqref{eq:rq-contraction-1}, there exist constants $C>0, r_0>0$ such that for all $a\in [a_L,a_U]$ and all $0<r\leq r_0$:
\begin{align}
\label{eq:rq-contraction-rate-5}
-\log \Pi_a\left(\lVert W^a\rVert_{\infty} \leq r\right) \geq C(a\vee 1)^{d}r^{-\frac{d}{\nu}}.
\end{align}
For all $a\in [a_L, a_U]$,~\eqref{eq:rq-contraction-rate-5} implies:
\begin{align}
\label{eq:rq-contraction-rate-6}
\inf_{a\in [a_L, a_U]}\Pi_a\left(\lVert W^{a}\rVert_{\infty} \leq r\right) \geq e^{-C(a_U\vee 1)^d r^{-\frac{d}{\nu}}}.
\end{align}
Combine~\eqref{eq:rq-contraction-rate-2},~\eqref{eq:rq-contraction-rate-4}, and~\eqref{eq:rq-contraction-rate-6} to get:
\begin{align}
\label{eq:rq-contraction-rate-7}
\Pi(E_r) \geq (a_U-a_L) \underline{g}e^{-\frac{R^2}{2} - C(a\vee 1)^{d}r^{-\frac{d}{\nu}}}.
\end{align}

Recall $r = C_{\star}\epsilon$ with $C_{\star} := \sqrt{D\alpha/4c_L}$. By~\eqref{eq:rq-contraction-rate-1} and~\eqref{eq:rq-contraction-rate-7}:
\begin{align*}
\Pi\left(B(\pmb\theta_0, \epsilon)\right) \geq (a_U-a_L) \underline{g} e^{-\frac{1}{2}R^{2} - C(a\vee 1)^{d} (C_{\star}\epsilon)^{-\frac{d}{\nu}}}.
\end{align*}
Equivalently:
\begin{align}
\label{eq:rq-contraction-rate-8}
-\log \Pi\left(B(\pmb\theta_0, \epsilon)\right) \leq \frac{R^2}{2} - \log((a_U-a_L)\underline{g}) + C(a\vee 1)^{d}C_{\star}^{-\frac{d}{\nu}}\epsilon^{-\frac{d}{\nu}}.
\end{align}
Therefore, from~\eqref{eq:rq-contraction-rate-8}, pick $\epsilon = \epsilon_t$ to make the dominant term, $C(a\vee 1)^{d}C_{\star}^{-d/\nu}\epsilon^{-d/\nu}$, comparable to $t\epsilon_t^2$, i.e., $t\epsilon_t^2 \asymp \epsilon_t^{-d/\nu} \iff t^{-\nu / (2\nu+d)}$. Choose $K>0$ large enough so that, $t\epsilon_t^2 \geq 2(R^2/2 - \log((a_U-a_L)\underline{g}))$ and $C(a\vee 1)^{d}C_{\star}^{-d/\nu}\epsilon^{-d/\nu} \leq D\alpha t\epsilon_t^{2}/8$ for all large $t$. Then,~\eqref{eq:rq-contraction-rate-8} yields the prior mass condition, $-\log \Pi(B(\pmb\theta_0, \epsilon_t)) \lesssim t\epsilon_t^{2}$, with the $\alpha$-posterior contraction rate (following~\cite{bhattacharya2019bayesian}[Corollary 3.5]):
\begin{align}
\label{eq:rq-contraction-rate-9}
\epsilon_t \asymp t^{-\frac{\nu}{2\nu+d}}.
\end{align}
\end{proof}

\section{Bound on information gain ($\gamma_T$) for Rational Quadratic Kernel}
Denote by $T_k:L^2(\mathcal X)\to L^2(\mathcal X)$ the associated kernel
integral operator:
\[
(T_k f)(x) = \int_{\mathcal X} k(x,x') f(x') \, dx'.
\]
Let $\{\lambda_j\}_{j\ge1}$ denote the eigenvalues of $T_k$ in non-increasing order.

\begin{lemma}
\label{lem:rq-gamma}
Let $k_{\mathrm{RQ}}(\cdot, \cdot)$ be the rational quadratic kernel:
\[
k_{\mathrm{RQ}}(x,x')
=
\left(1 + \frac{\|x-x'\|^2_2}{2\nu\ell^2}\right)^{-\nu},
\quad \ell>0, \nu>0.
\]
Then the maximum information gain satisfies:
\[
\gamma_T(k_{\mathrm{RQ}})
=
\mathcal{\tilde O}\!\left(T^{\frac{d}{2\nu+d}}\right).
\]
\end{lemma}

\begin{proof}
Using Bochner's theorem, the rational quadratic kernel is stationary and admits the following high-frequency spectral density
representation:
\begin{align}
\label{eq:rq-gamma-1}
S(\omega) =\int \left(1 + \frac{\|\tau\|^2_2}{2\nu\ell^2}\right)^{-\nu} e^{-2\pi i \omega^{\top} \tau} d\tau
\propto
\left(1 + \frac{\ell^2\|\omega\|^2_2}{2\nu}\right)^{-(\nu+ d/2)},
\end{align}
which implies the polynomial tail bound:
\begin{align}
\label{eq:rq-gamma-2}
S(\omega) = \mathcal O\!\left(\|\omega\|^{-(2\nu+ d)}_2\right),
\quad \text{as } \|\omega\|_2\to\infty.
\end{align}

For stationary kernels on compact domains, polynomial decay of the spectral
density implies polynomial decay of the eigenvalues of the associated kernel
integral operator (e.g., see standard results in~\cite{bach2017equivalence} relating spectral densities and
Mercer eigenvalues).
In particular, the eigenvalues of $T_{k_{\mathrm{RQ}}}$ satisfy:
\begin{align}
\label{eq:rq-gamma-3}
\lambda_j
=
\mathcal O\!\left(j^{-(1 + 2\nu/d)}\right).
\end{align}

Applying the general information gain bound for kernels with polynomial
eigenvalue decay~\citep{vakili21a}, we obtain:
\begin{align}
\label{eq:rq-gamma-4}
\gamma_T
\;\le\;
C \sum_{j\ge1} \log(1 + T \lambda_j)
=
\mathcal{\tilde O}\!\left(T^{\frac{d}{2\nu+ d}}\right),
\end{align}
which completes the proof.
\end{proof}

\end{document}